\documentclass[11pt,reqno]{amsart}

\usepackage{amsfonts,amssymb,amsmath,xcolor}

\usepackage[cp1251]{inputenc}
\usepackage[T2B]{fontenc}
\usepackage[english]{babel}
\usepackage[mathscr]{eucal}
\usepackage{calrsfs}

\usepackage{tikz}
\usepackage{ragged2e}

\usepackage{tikz}
\usepackage[all]{xy}

\newtheorem{theorem}{Theorem}[section]

\newtheorem{definition}[theorem]{Definition}

\theoremstyle{remark}
\newtheorem{remark}[theorem]{Remark}
\newtheorem{example}[theorem]{Example}

\def\loc{\text{\rm loc}}

\usepackage[left=2.0cm,right=2.2cm,
    top=2.5cm,bottom=2.5cm]{geometry}

\begin{document}

\centerline{\bf IMAGES OF INTEGRATION OPERATORS IN WEIGHTED
FUNCTION SPACES}

\smallskip\begin{center}
{\bf Elena P. Ushakova}\footnote{The results of Sections~4 and 5 were performed at Steklov Mathematical Institute of Russian Academy of Sciences under financial support of the Russian Science Foundation (project 19-11-00087). The rest part of the paper was carried out within the framework of the State Tasks of Ministry of Education and Science of Russian Federation for V.A. Trapeznikov Institute of Control Sciences of Russian Academy of Sciences and Computing Center of Far--Eastern Branch of Russian Academy of Sciences, it was also partially supported by the Russian Foundation for Basic Research (project 19--01--00223).}\end{center}

\smallskip
\centerline{\textit {V.A. Trapeznikov Institute of Control Sciences of RAS, Moscow, Russia;}}
\centerline{\textit {Steklov Mathematical Institute of RAS, Moscow, Russia;} }
\centerline{\textit {Computing Center of FEB RAS, Khabarovsk, Russia.}}
\centerline{E--mail address: elenau@inbox.ru}

\medskip



{\small {\bf Abstract.} Images of integration operators of natural orders are considered as elements of Besov and Triebel--Lizorkin 
 spaces with local Muckenhoupt weights on $\mathbb{R}^N$. The results connect entropy and approximation numbers of embedding operators with the same characteristics of the integration operators.}

\smallskip
\section{Introduction}
Let $N\in\mathbb{N}$. For $k\in\{1,\ldots,N\}$ and $c_k\in\mathbb{R}$ denote $\Omega^+_k:=[c_k,+\infty)$ and $\Omega^-_k:=(-\infty,c_k]$, and put ${y}_{x_k}:=(y_1,\ldots,y_{k-1},x_k,y_{k+1},\ldots,y_N)\in\mathbb{R}^N$, that is $y=(y_1,\ldots,y_k,\ldots,y_N)=y_{y_k}$.

Fix $l,n\in\{1,\ldots,N\}$ and assume $f\in L_1^\mathrm{loc}(\mathbb{R}^N)$. Following the notations adopted in \cite{SKM}, denote by
\begin{equation}\label{left}
I_{\Omega_l^+}^{\alpha_l} f({y}_{x_l}):=\frac{1}{\Gamma(\alpha
_l)}\int_{c_l}^{x_l}{(x_l-y_l)^{\alpha_l-1}}{f({y})}\,dy_l\qquad(x_l\in\Omega_l^+)
\end{equation} and
\begin{equation}\label{right}
I_{\Omega_n^-}^{\alpha_n} f({y}_{x_l}):=\frac{1}{\Gamma(\alpha_n)}\int_{x_n}^{c_n}{(y_n-x_n)^{\alpha_n-1}}{f({y})}\,dy_n\qquad(x_n\in\Omega_n^-)
\end{equation} the left-- and right--hand side Riemann--Liouville operators of (positive) orders $\alpha_l$ and $\alpha_n$, respectively. For natural $\alpha_k=m_k$, $k\in\{l,n\}$,  $I_{\Omega_l^+}^{m_l}$ and $I_{\Omega_n^-}^{m_n}$ are integration operators of natural orders $m_l$ and $m_n$.

Given $N\in\mathbb{N}$ put $|m|:=\sum_{k=1}^Nm_k$ for a set of multi--indices $m=(m_1,\ldots,m_N)$, $m_k\in\mathbb{N}_0:=\mathbb{N}\cup\{0\}$. Then, using the notation $I^0_{\Omega_k^0}f(y):=f(y_{y_k})$, $k\in\{1,\ldots,N\}$, we can write, formally, if $N>1$ that $I_{\Omega_l^+}^{m_l}={I^0_{\Omega_1^0}\ldots I_{\Omega_{l-1}^0}^0}I_{\Omega_l^+}^{m_l}{I^0_{\Omega_{l+1}^0}\ldots I_{\Omega_N^0}^0}=:I_{\Omega^{(0,\ldots,+,\ldots,0)}}^{|m|}$ for $m=(0,\ldots,0,m_l,0\ldots,0)$ and, similarly, that $I_{\Omega_n^-}^{m_n}={I^0_{\Omega_1^0}\ldots I_{\Omega_{n-1}^0}^0}I_{\Omega_n^-}^{m_n}{I^0_{\Omega_{n+1}^0}\ldots I_{\Omega_N^0}^0}=:I_{\Omega^{(0,\ldots,-,\ldots,0)}}^{|m|}$ with $m=(0,\ldots,0,m_n,0\ldots,0)$.

Let $0<p<\infty$, $0<q\le\infty$ and $s\in\mathbb{R}$. For functions $f$, locally integrable on $\mathbb{R}^N$, we study images of various combinations of operators \eqref{left} and/or \eqref{right} --- integration operators $I^{|m|}_{\Omega^\star}$ of orders $m=(m_1,\ldots,m_N)$ --- belonging to weighted Besov $B_{pq}^{s,w}(\mathbb{R}^N)$ or Triebel--Lizorkin $F_{pq}^{s,w}(\mathbb{R}^N)$ spaces. Our consideration is reduced to weights $w$ from local Muckenhoupt classes $\mathscr{A}_p^\loc$, $1\le p\le\infty$. 
{For simplicity, we consider $w$ of product type, that is weights $w$, we operate with, admit factorisation of the form $w(x_1,\ldots,x_N)=w_1(x_1)\cdot\ldots\cdot w_N(x_N)$ with one variable functions $w_k(x_k)$, $k=1,\ldots,N$.
}
Symbol $\star=(\star_1,\ldots,\star_N)$ with $\star_k\in\{0,+,-\}$, $k=1,\ldots,N$, in $I^{|m|}_{\Omega^\star}$ indicates ordering of $I^0_{\Omega_k^0}$ and $I_{\Omega_k^+}^{m_k}$ with $I_{\Omega_k^-}^{m_k}$ in a combination of integration operators \eqref{left} and/or \eqref{right} producing an $I^{|m|}_{\Omega^\star}$, where $\Omega^\star:=\Omega_1^{\star_1}\times\ldots\times\Omega_N^{\star_N}$ with $\Omega_k^0\equiv\mathbb{R}$. We put $\star_k=0$, $\star_k=+$ and $\star_k=-$ for $I^0_{\Omega_k^0}$, $I_{\Omega_k^+}^{m_k}$ and $I_{\Omega_k^-}^{m_k}$, respectively. Obviously, $m_k$, $k=1,\ldots,N$, in $m=(m_1,\ldots,m_N)$ must correspond to the orders of operators $I^0_{\Omega_k^0}$, $I_{\Omega_k^+}^{m_k}$ and $I_{\Omega_k^-}^{m_k}$ standing at related places, where we set $m_k=0$ for $I^0_{\Omega_k^0}$. For simplicity, we assume $f(x_{x_l})\equiv 0$ for $x_l\in(-\infty,c_l)$ and/or $f(x_{x_n})\equiv 0$ for $x_n\in(c_n,+\infty)$.
Say, in the case $N=2$, the following mappings come to our consideration:
\begin{gather}
I_{\Omega^{(+,+)}}^{|m|}f({x})=I_{\Omega_1^+}^{m_1}I_{\Omega_2^+}^{m_2} f({x})=\frac{1}{\Gamma(m
_1)\Gamma(m
_2)}\int_{c_1}^{x_1}\int_{c_2}^{x_2}\frac{f(y)\,dy_2\,dy_1}{(x_2-y_2)^{1-m_2}(x_1-y_1)^{1-m_1}}\quad(x\in\Omega_1^+\times\Omega_2^+),
\nonumber\\
I_{\Omega^{(-,-)}}^{|m|}f({x})=I_{\Omega_1^-}^{m_1}I_{\Omega_2^-}^{m_2} f({x})=\frac{1}{\Gamma(m
_1)\Gamma(m
_2)}\int_{x_1}^{c_1}\int_{x_2}^{c_2}\frac{f(y)\,dy_2\,dy_1}{(y_2-x_2)^{1-m_2}(y_1-x_1)^{1-m_1}}\quad(x\in\Omega_1^-\times\Omega_2^-),\nonumber\\
I_{\Omega^{(+,-)}}^{|m|}f({x})=I_{\Omega_1^+}^{m_1}I_{\Omega_2^-}^{m_2} f({x})=\frac{1}{\Gamma(m_1)\Gamma(m_2)}\int_{c_1}^{x_1}\int_{x_2}^{c_2}\frac{f(y)\,dy_2\,dy_1}{(y_2-x_2)^{1-m_2}(x_1-y_1)^{1-m_1}}\quad(x\in\Omega_1^+\times\Omega_2^-),\nonumber\\
I_{\Omega^{(-,+)}}^{|m|}f({x})=I_{\Omega_1^-}^{m_1}I_{\Omega_2^+}^{m_2} f({x})=\frac{1}{\Gamma(m_1)\Gamma(m_2)}\int_{x_1}^{c_1}\int_{c_2}^{x_2}\frac{f(y)\,dy_2\,dy_1}{(x_2-y_2)^{1-m_2}(y_1-x_1)^{1-m_1}}\quad(x\in\Omega_1^-\times\Omega_2^+),\nonumber\\\label{starEx}
I_{\Omega^{(\pm,0)}}^{|m|} f(x_1,y_2)=I_{\Omega_1^\pm}^{m_1} I^0_{\Omega_2^0}f(x_1,y_2)\qquad (y_{x_1}\in\Omega_1^\pm\times\mathbb{R}),\\
I_{\Omega^{(0,\pm)}}^{|m|} f(y_1,x_2)=I^0_{\Omega_1^0}I_{\Omega_2^\pm}^{m_2} f(y_1,x_2)\qquad (y_{x_2}\in\mathbb{R}\times\Omega_2^\pm).\nonumber
\end{gather}

Definitions of function spaces $B_{pq}^{s,w}(\mathbb{R}^N)$ and $F_{pq}^{s,w}(\mathbb{R}^N)$ with local Muckenhoupt weights $w$ are collected in \S\,\ref{WFS}.
In \S\,\ref{Emb} we find relations between norms of images $I_{\Omega^\star}^{|m|}f$ of integration operators $I_{\Omega^\star}^{|m|}$ and their pre--images $f$ in $B_{pq}^{s,w}(\mathbb{R}^N)$ and $F_{pq}^{s,w}(\mathbb{R}^N)$, and apply these results in \S\,\ref{EnAp} to estimates of the entropy and approximation numbers of some integration operators (see definitions of the numbers in \S\,\ref{EnAp}).

Connections between images and pre--images of integration (differentiation) operators in $B_{pq}^{s,w}(\mathbb{R}^N)$ and $F_{pq}^{s,w}(\mathbb{R}^N)$ have been studied in \cite[Theorem 2.3.8]{Tr1}, \cite{N}, \cite[Theorem 2.20]{R}, \cite[\S\,4]{IS}, \cite[p.\,23]{Tr6}.

Let $D^m$ stand for derivatives. To simplify notations, we write $A_{pq}^{s,w}(\mathbb{R}^N)$ instead of $B_{pq}^{s,w}(\mathbb{R}^N)$ or $F_{pq}^{s,w}(\mathbb{R}^N)$. It follows from a modification of \cite[Theorem 2.20]{R} (see also \cite[Prop. 4.2]{IS}) of the form \begin{equation}\label{DifInt}\|D^m \mathscr{F}\|_{A_{pq}^{s-|m|,w}(\mathbb{R}^N)}\lesssim \|\mathscr{F}\|_{A_{pq}^{s,w}(\mathbb{R}^N)}\qquad (s\in\mathbb{N}) \end{equation} that, in particular, $\mathscr{F}\in A_{pq}^{s,w}(\mathbb{R}^N)$ yields $D^m\mathscr{F}\in A_{pq}^{s-|m|,w}(\mathbb{R}^N)$, where $\mathscr{F}$ denotes the image of an $I_{\Omega^\star}^{|m|}$. The main result of this work is a type of reverse implications for such a particular case of $\mathscr{F}$ 
\begin{equation}\label{DifInt'}\|\mathscr{F}\|_{B_{pq}^{s,w}(\mathbb{R}^N)}\lesssim C_{u,w}\|D^m \mathscr{F}\|_{B_{pq}^{s+|m|,u}(\mathbb{R}^N)} \end{equation}
in Besov spaces (see Theorems \ref{ImagesA}, \ref{ImagesB} and \ref{ImagesComb}). From them, analogous estimates in Triebel--Lizorkin spaces can be obtained by the chain of embeddings $B_{pq_\ast}^{s,v}(\mathbb{R}^N)\hookrightarrow F_{pq}^{s,v}(\mathbb{R}^N) \hookrightarrow B_{pq^\ast}^{s,v}(\mathbb{R}^N)$ for $v=u$ or $v=w$, where $q_\ast\le\min\{p,q\}\le\max\{p,q\}\le q^\ast$. Observe that \eqref{DifInt} was performed in \cite{R} for $A_{pq}^{s,w}(\mathbb{R}^N)$ and $A_{pq}^{s-|m|,w}(\mathbb{R}^N)$ with natural $s$ on test functions $\mathscr{F}$. {{In \S\,\ref{Emb} we give shortly a proof of this inequality 
for $I_{\Omega^\star}^{|m|}f$ belonging to $A_{pq}^{s,w}(\mathbb{R}^N)$ with arbitrary $s\in\mathbb{R}$.} The proof method, we use for this, is valid for arbitrary $\mathscr{F}\in A_{pq}^{s,w}(\mathbb{R}^N)$. Our results concerning the reverse implications \eqref{DifInt'} infers exploiting characteristics for weighted discrete Hardy inequalities described in settings of general measures in e.g. \cite[\S\,1]{PSU}.}

Instruments of our work are properties of $B-$ splines and \cite[Theorem 4.10]{RMC*}, we remind this material in \S\S\,\ref{BL}--\ref{Decom}. Theorem 4.10 in \cite{RMC*} uses localised spline wavelet bases of Battle--Lemari\'{e} type, which elements have explicit forms. In \S\,\ref{BL} we recall their exact formulae as we need them for establishing our main result.

Throughout the paper relations of the type $A\lesssim B$ mean that $A\le cB$ with
some constant $0<c<\infty$ depending, possibly, on number parameters. We write
$A\approx B$ instead of $A\lesssim B \lesssim A$ and $A\simeq B$ instead of $A=cB$. We stand $\mathbb Z$, $\mathbb N$ and $\mathbb R$ for integers, natural and real numbers, respectively, and $\mathbb{C}$ for the complex plane. By $\mathbb{N}_0$ we denote the set $\mathbb{N}\cup\{0\}$, and by $\mathbb{R}^N$, $N\in\mathbb{N}$, --- the Euclidean $N-$space.
The symbol $dx$ stands for the $N-$dimensional Lebesgue measure, $\Gamma(\cdot)$ --- for the Gamma function, 
 $[s]$ --- for the integer part of $s\in\mathbb{R}$.
We put $r':=
{r}/({r-1})$ if $0<r<\infty$ and $r'=1$ for $r=\infty$. 
We say that $f\in L_1^\mathrm{loc}(\mathbb{R}^N)$ if $f\in L_1(\Omega)$ for every compact subset $\Omega$ of $\mathbb{R}^N$. Marks $:=$ and
$=:$ are used for introducing new quantities. We abbreviate $h(\Omega):=\int_\Omega h(x)\, dx$,
where {$\Omega \subset\mathbb{R}$} is some bounded measurable set.

\section{Weighted function spaces}\label{WFS}
\subsection{Local Muckenhoupt weights $\mathscr{A}_\infty^\loc$} A weight $w$ is a locally  integrable function on $\mathbb{R}^N$, positive almost everywhere (a.e.).

Let $Q$ be a cube $Q\subset\mathbb{R}^N$ with sides parallel to the coordinate axes, and $|Q|$ stands for its volume.

Let $L_r(\mathbb{R}^N)$ with $0<r\le\infty$ denote the Lebesgue space of all measurable functions $f$ on $\mathbb{R}^N$ quasi--normed by
$\|f\|_{L_r(\mathbb{R}^N)}:=\bigl(\int_{\mathbb{R}^N}|f(x)|^r\,dx\bigr)^{{1}/{r}}$ with the usual modification if $r=\infty$.

\begin{definition}\label{defin}{\rm (\cite[\S~1.1]{R})\\ (i) A weight $w$ belongs to the class $\mathscr{A}_p^\loc$, $1<p<\infty$, of \textit{local Muckenhoupt weights} if \begin{equation}\label{Mup'}
\mathscr{A}_p^\loc(w):=\sup_{|Q|\le 1}\frac{w(Q)}{|Q|}
\biggl(\frac{1}{|Q|}\int_{Q} w^{1-p'}\biggr)^{\frac{p}{p'}}<\infty;\end{equation}\\
(ii) $w\in\mathscr{A}_1^\loc$ if \begin{equation}\label{Mu1'}
\mathscr{A}_1^\loc(w):=\sup_{|Q|\le 1}\frac{w(Q)}{|Q|}\,\|1/w\|_{L_\infty(Q)}<\infty;\end{equation}\\
(iii) we say that $w\in\mathscr{A}_\infty^\loc$ if $w\in \mathscr{A}_p^\loc$ for some $1\le p<\infty$, that is $\mathscr{A}_\infty^\loc:=\bigcup_{p\ge 1}\mathscr{A}_p^\loc.$
}\end{definition} Suprema in \eqref{Mup'} and \eqref{Mu1'} are taken over all cubes $Q\subset\mathbb{R}^N$ with $|Q|\le 1$. The class $\mathscr{A}_\infty^\loc$ is stable with respect to translation, dilation and multiplication {by} a positive scalar.
Besides,\\ --- if $w\in\mathscr{A}_p^\loc$, $1<p<\infty$, then $w^{-p'/p}\in\mathscr{A}_{p'}^\loc$;\\
--- if $w\in\mathscr{A}_p^\loc$, $1<p\le\infty$, then there exists constant $c_{w,N}>0$ such that the inequality \begin{equation}\label{delta2}w(Q_t)\le \exp(c_{w,N}\, t) w(Q)\end{equation} holds for all cubes $Q$ with $|Q|=1$ and their concentric cubes $Q_t$ with side lengths $l(Q)$ and $l(Q_t)$, respectively, satisfying the relation $l(Q_t)=t\cdot l(Q)$ for $t\ge 1$;\\ --- if $1<p_1<p_2\le\infty$ then $\mathscr{A}_{p_1}^\loc\subset \mathscr{A}_{p_2}^\loc$;\\ --- if $w\in\mathscr{A}_p^\loc$ then there exists some number $r<p$ such that $w\in\mathscr{A}_r^\loc$.

The last property {emerges} in the following definition of the special number \begin{equation}\label{r_w}r_0:= \inf\{r\ge 1\colon w\in\mathscr{A}_r^\loc\}<\infty, \qquad w\in\mathscr{A}_\infty^\loc.\end{equation} Definition \ref{defin}(iii) implies that if $\mathscr{A}_{\infty}^\loc$ then $w\in\mathscr{A}_{p}^\loc$ for some $p<\infty$. 

One of the most prominent examples of a local Muckenhoupt weight is the following:
$$\mathscr{A}_p^\loc\ni w(x)=\begin{cases}|x|^\alpha, & |x|\le 1,\\
\exp(|x|-1), & |x|>1,\end{cases}\quad\textrm{where}\quad \begin{cases} -N<\alpha<N(p-1), & p>1,\\ -N<\alpha\le 0, &p=1.\end{cases}$$
\noindent More information and examples of local Muckenhoupt weights can be found in \cite{R} and \cite{Ma,WM}.

\subsection{Function spaces $B_{p,q}^{s,w}(\mathbb{R}^N)$ and $F_{p,q}^{s.w}(\mathbb{R}^N)$ with $w\in\mathscr{A}_\infty^\loc$}\label{E} For $0<p<\infty$ and a weight $w$ on $\mathbb{R}^N$ we denote $L_p(\mathbb{R}^N,w)$ the weighted Lebesgue space quasi--normed by $\|f\|_{L_p(\mathbb{R}^N,w)}:=\|w^{1/p}f\|_{L_p(\mathbb{R}^N)}$ with usual modification in the case $p=\infty$. Let $\mathscr{D}(\mathbb{R}^N)$ denote the space of
all compactly supported $C^\infty(\mathbb{R}^N)$ functions equipped with the usual topology.

For the definitions of the unweighted Besov $B_{pq}^s(\mathbb{R}^N)$ and Triebel--Lizorkin $F_{pq}^s(\mathbb{R}^N)$ spaces we refer to \cite{Tr1, Tr2}. To incorporate the $\mathscr{A}_\infty^\loc$ class of weights into the theory of weighted function spaces V.S. Rychkov exploited a class $\mathscr{S}'_e(\mathbb{R}^N)$ of distributions (see \cite{R} and e.g. \cite{Sch,Sch'}).
\begin{definition}{\rm
The space
$\mathscr{S}_e(\mathbb{R}^N)$ of test functions consists of all
$C^\infty(\mathbb{R}^N)$ functions ${\varphi}$ satisfying $${q}_{\mathbf{N}}({\varphi}):=\sup_{x\in\mathbb{R}^N}\mathrm{e}^{\mathbf{N}|x|}\sum_{|\gamma|\le \mathbf{N}}|D^\gamma{\varphi}(x)|<\infty\qquad\textrm{for all}\quad\mathbf{N}\in\mathbb{N}_0.$$ }\end{definition} The space
$\mathscr{S}_e(\mathbb{R}^N)$ is equipped with the locally convex topology defined by the system of the semi--norms ${q}_\mathbf{N}$. The set $\mathscr{S}_e(\mathbb{R}^N)$ is a complete locally convex space. It holds $\mathscr{D}(\mathbb{R}^N)\hookrightarrow\mathscr{S}_e(\mathbb{R}^N)$, moreover, the space $\mathscr{D}(\mathbb{R}^N)$ is dense in $\mathscr{S}_e(\mathbb{R}^N)$. If $w\in \mathscr{A}_\infty^\loc$ then $\mathscr{S}_e(\mathbb{R}^N)\hookrightarrow L_p(\mathbb{R}^N,w)$ for any $0<p<\infty$ (see \cite{Sch}, \cite{R} and \cite[\S~3]{Ma}).

By $\mathscr{S}'_e(\mathbb{R}^N)$ we denote the strong dual space of  $\mathscr{S}_e(\mathbb{R}^N)$. The class $\mathscr{S}'_e(\mathbb{R}^N)$ can be identified with a subset of the collection $\mathscr{D}'(\mathbb{R}^N)$ of all distributions $f$ on $\mathscr{D}(\mathbb{R}^N)$ for which the estimate $$|\langle f,{\varphi}\rangle|\le C\sup\Bigl\{\bigl|D^\gamma {\varphi}(x)\bigr|\mathrm{e}^{\mathbf{N}|x|}\colon x\in\mathbb{R}^N,\,|\gamma|\le \mathbf{N}\Bigr\}$$ holds for all ${\varphi}\in\mathscr{D}(\mathbb{R}^N)$ with some constants $C$ and $\mathbf{N}$ depending of $f$. Such a distribution $f$ can be extended to a continuous functional on $\mathscr{S}_e(\mathbb{R}^N)$.

For ${\varphi},{\psi}\in\mathscr{S}_e(\mathbb{R}^N)$ their convolution of the form $${\varphi}\ast{\psi}(x)=:\int_{\mathbb{R}^N}{\varphi}(x-y){\psi}(y)\,dy,\quad x\in\mathbb{R}^N,$$
belongs to $\mathscr{S}_e(\mathbb{R}^N)$. Besides, for $f\in\mathscr{S}'_e(\mathbb{R}^N)$ and ${\varphi}\in\mathscr{S}_e(\mathbb{R}^N)$ the related convolution $$f\ast{\varphi}(x)=:\langle f,{\varphi}(x-\cdot)\rangle,\quad x\in\mathbb{R}^N,$$
is a $C^\infty(\mathbb{R}^N)$ function of at most exponential growth.

Let a function ${\varphi}_0\in\mathscr{D}(\mathbb{R}^N)$ be such that \begin{equation}\label{Iphi1}\int_{\mathbb{R}^N}{\varphi}_0(x)\,dx\not=0.\end{equation} Put
\begin{equation}\label{Iphidef}{\varphi}(x)={\varphi}_0(x)-2^{-N}{\varphi}_0(x/2)\end{equation} and let ${\varphi}_d(x):=2^{(d-1)N}{\varphi}(2^{d-1}x)$ for $d\in\mathbb{N}$. One can find ${\varphi}_0$ such that \begin{equation}\label{Iphi0}\int_{\mathbb{R}^N}x^\gamma{\varphi}(x)\,dx=0\end{equation} for any multi--index $\gamma\in\mathbb{N}^N_0$, $|\gamma|\le \Gamma$, where $\Gamma\in\mathbb{N}$ is fixed. We write $\Gamma=-1$ if \eqref{Iphi0} does not hold.

\begin{definition}\label{Def2}{\rm (\cite{R})
Let $0<p<\infty$, $0<q\le\infty$, $s\in \mathbb{R}$ and $w\in\mathscr{A}_\infty^\loc$. Let a function ${\varphi}_0\in\mathscr{D}(\mathbb{R}^N)$ satisfy \eqref{Iphi1} and ${\varphi}$ of the form \eqref{Iphidef} satisfy \eqref{Iphi0} with $|\gamma|\le \Gamma$, where $\Gamma\ge [s]$. We define\\
(i) the weighted Besov space $B_{pq}^{s,w}(\mathbb{R}^N)$ to be the set of all $f \in \mathscr{S}'_e(\mathbb{R}^N)$ such that the quasi--norm
\begin{equation}\label{Bspq}
  \|f\|_{B_{pq}^{s,w}(\mathbb{R}^N)}: =\biggl(\sum_{d=0}^\infty 2^{dsq}\bigl\| {\varphi}_d\ast f\bigr\|_{L_p(\mathbb{R}^N,w)}^q\biggr)^{\frac{1}{q}}
\end{equation}
(with the usual modification if $q=\infty$) is finite; and\\
(ii) the weighted Triebel--Lizorkin space $F_{pq}^{s,w}(\mathbb{R}^N)$ as the set of all $f \in \mathscr{S}'_e(\mathbb{R}^N)$ with the finite quasi--norm
\begin{equation}\label{Fspq}
  \|f\|_{F_{pq}^{s,w}(\mathbb{R}^N)} =\biggl\|\biggl(\sum_{d=0}^\infty 2^{dsq} \bigl|{\varphi}_d\ast f\bigr|^q\biggr)^{\frac{1}{q}} \biggr\|_{L_p(\mathbb{R}^N,w)}
\end{equation}
(usually modified in the case $q=\infty$).
}
\end{definition}

Definitions of the above spaces independent of the choice of ${\varphi}_0$, up to equivalence of quasi--norms. 

To simplify the notation we write $A^{s,w}_{pq}(\mathbb{R}^N)$ instead of $B_{pq}^{s,w}(\mathbb{R}^N)$ or $F_{pq}^{s,w}(\mathbb{R}^N)$. The spaces $A^{s,w}_{pq}(\mathbb{R}^N)$ have properties similar to the unweighted $B^{s}_{pq}(\mathbb{R}^N)$ and $F^{s}_{pq}(\mathbb{R}^N)$ and to other distribution spaces with less general weights than $\mathscr{A}_\infty^\loc$ (see \cite{IS,Ma,R,WBan,WM} for details and other properties of $A_{pq}^{s,w}(\mathbb{R}^N)$).

\section{Spline wavelet bases of Battle--Lemari\'{e} type} \label{BL}

For a function $g\in L_1(\mathbb{R}^N)$ its Fourier transform has the form
\begin{equation} \label{FourierS}
\widehat{g}(\omega)=(2\pi)^{-N/2}\int_{\mathbb{R}^N}\mathrm{e}^{-i\omega x}g(x)\,dx, \qquad \omega\in\mathbb{R}^N.
\end{equation}

Put $B_0=\chi_{[0,1)}$ and define B--spline $B_n$ of order $n\in\mathbb{N}$ \begin{equation}\label{Bndef}
B_n(x):=(B_{n-1}\ast B_0)(x)=\int_0^1 B_{n-1}(x-t)\,dt=\frac{x}{n}B_{n-1}(x)
+\frac{n+1-x}{n}B_{n-1}(x-1).
\end{equation} {It is known \cite{Chui} that $B_n$ is continuous and $n-$times a.e. differentiable function on $\mathbb{R}$} {with ${\rm supp}\,B_n=[0,n+1]$. Besides, $B_n(x)>0$ for all $x\in(0,n+1)$ and the restriction of $B_n$ to each $[m,m+1]$,} {$m=0,\ldots,n$, is a polynomial of degree $n$. The function $B_n(x)$ is symmetric about $x=(n+1)/2$.}

\label{MRA} Let $\mathscr{V}_d$, $d\in\mathbb{Z}$, denote the $L_2(\mathbb{R})-$closure of the linear span of the system $\bigl\{B_{n}(2^d\cdot-\tau)\colon \tau\in\mathbb{Z}\bigr\}$. The spline spaces $\mathscr{V}_d$, $d\in\mathbb{Z}$, constitute multiresolution analysis of $L_2(\mathbb{R})$ \cite{Chui,W} in the sense that \begin{itemize}
\item[{\rm (i)}] $\ldots\subset \mathscr{V}_{-1}\subset \mathscr{V}_0\subset \mathscr{V}_1\subset\ldots$;
\item[{\rm (ii)}] $\mathrm{clos}_{L_2(\mathbb{R})}\Bigl(\bigcup_{d\in\mathbb{Z}} \mathscr{V}_d\Bigr)=L_2(\mathbb{R})$;
\item[{\rm (iii)}] $\bigcap_{d\in\mathbb{Z}} \mathscr{V}_d=\{0\}$;
\item[{\rm (iv)}] for each $d$ the $\bigl\{B_{n}(2^d\cdot-\tau)\colon\tau\in\mathbb{Z}\bigr\}$ is an unconditional (but not orthonormal) basis of $\mathscr{V}_d$.\end{itemize}
Further, there are the orthogonal complementary subspaces $\ldots, \mathscr{W}_{-1},\mathscr{W}_0,\mathscr{W}_1,\ldots$ such that
\begin{itemize}
\item[{\rm (v)}] $\mathscr{V}_{d+1}=\mathscr{V}_d\oplus \mathscr{W}_d$ for all $d\in\mathbb{Z}$,
where $\oplus$ stands for $\mathscr{V}_d\perp \mathscr{W}_d$ and $\mathscr{V}_{d+1}=\mathscr{V}_d+\mathscr{W}_d$.
\end{itemize} Wavelet subspaces $\mathscr{W}_d$, $d\in\mathbb{Z}$, related to the spline $B_n$, are also generated by some basis functions ({\it wavelets}) in the same manner as the spline spaces $\mathscr{V}_d$, $d\in\mathbb{Z}$, are generated by the spline $B_n$.
The whole structure above means that {for any fixed} $\boldsymbol{k}\in\mathbb{Z}$ the system $\bigl\{B_{n,\boldsymbol{k}}(\cdot-\tau):=B_n(\cdot-\boldsymbol{k}-\tau)\colon \tau\in\mathbb{Z}\bigr\}$ generates multiresolution analysis $\mathrm{MRA}_{B_{n,\boldsymbol{k}}}$ of $L_2(\mathbb{R})$, and $\mathrm{MRA}_{B_{n,\boldsymbol{k}}}=\mathrm{MRA}_{B_n}$ for any $\boldsymbol{k}\in\mathbb{Z}$.

\smallskip
Fix $n\in\mathbb{N}$. For each $j=1,\ldots, n$ we introduce some $\alpha_{j}(n)>1$ and define $r_{j}(n):=(2\alpha_j(n)-1)-2\sqrt{\alpha_j(n)(\alpha_j(n)-1)}\in(0,1)$. The collection
$\alpha_j(n)$, $j=1,\ldots,n$, and, respectively, the set of numbers $r_j(n)$, $j=1,\ldots,n$, are uniquely defined {dependent on} $n$ (see \cite[\S\,2, p. 179]{JMAA}) so that the sequence $\bigl\{-r_j(n)\bigr\}_{j=1}^n$ is formed by the roots of Euler's polynomial \begin{equation}\label{ortho}\mathbb{P}_{n}(\omega):=\sum_{m\in\mathbb{Z}}\Bigl|\widehat{B}_n(\omega+2\pi m)\Bigr|^2=\frac{1}{2^{2n}{\alpha_1(n)t_1(n)\ldots\alpha_n(n)t_n(n)}}\bigl|1+\mathrm{e}^{\pm i\omega}t_1(n)\bigr|^{2}\ldots\bigl|1+\mathrm{e}^{\pm i\omega}t_n(n)\bigr|^{2},\end{equation}
where $t_j(n)\in\{r_j(n),1/r_j(n)\}$, $j=1,\ldots,n$. It holds $\mathbb{P}_{n,\boldsymbol{k}}(\omega):=\sum_{m\in\mathbb{Z}}\Bigl|\widehat{B}_{n,\boldsymbol{k}}(\omega+2\pi m)\Bigr|^2=\mathbb{P}_{n}(\omega)$.

Orthogonalisation process of $B$--splines \eqref{Bndef} results in other scaling functions than $B_n$, named after G. Battle \cite{B,B1} and P.G. Lemarie--Rieusset \cite{L}, whose integer translations form an orthonormal system within the multiresolution analysis of $L_2(\mathbb{R})$ generated by $B_{n,\boldsymbol{k}}$.
Relying on \eqref{ortho}, we define the $n-$th order Battle--Lemari\'{e} scaling function $\phi_{t_1,\ldots,t_n;\boldsymbol{k}}$ via its Fourier transform as follows (see \cite[\S\,3.2]{RMC}):
\begin{equation}\label{phi_nn}
\widehat{\phi}_{t_1,\ldots,t_n;\boldsymbol{k}}(\omega):=\frac{2^n\sqrt{\alpha_1(n)\,t_1(n)\ldots\alpha_n(n)\,t_n(n)}\,\widehat{B}_{n,\boldsymbol{k}}(\omega)}{(1+\mathrm{e}^{i\omega}t_1(n))\ldots
(1+\mathrm{e}^{i\omega}t_n(n))}.\end{equation}
One can choose scaling function $\phi_{t_1,\ldots,t_n;\boldsymbol{k}}$ with $t_j(n)=r_j(n)$ or $t_j(n)=r_j^{-1}(n)=1/r_j(n)$ for any $j=1,\ldots,n$. The parameter $\boldsymbol{k}\in\mathbb{Z}$ is fixed. It allows to start the construction of $\mathrm{MRA}_{B_{n,\boldsymbol{k}}}=\mathrm{MRA}_{{\phi}_{t_1,\ldots,t_n;\boldsymbol{k}}}$ with ${B}_{n,\boldsymbol{k}}$ centred at $(n+1)/2+\boldsymbol{k}$ for any $\boldsymbol{k}\in\mathbb{Z}$.

We say that $j\in J_r$ (or, alternatively, that $j\in J_{1/r}$), where $j\in\{1,\ldots,n\}$, if $t_j(n)=r_j(n)$ (or $t_j(n)=1/r_j(n)$). Let $c_r$ and $c_{1/r}$ denote cardinalities of the sets $J_{r}\subseteq \{1,\ldots,n\}$ and $J_{1/r}=\{1,\ldots,n\}\setminus J_r$, respectively.
Denote $\beta_n:=2^n\sqrt{\alpha_1(n)\,r_1(n)\ldots\alpha_n(n)\,r_n(n)}$ and observe that
for any $0<r<1$ \begin{equation}\label{ryadok}({\mathrm{e}^{\pm i\omega}r+1})^{-1}=\sum_{l=0}^\infty \left(-r\,\mathrm{e}^{\pm i\omega}\right)^l.\end{equation} Then, it follows from \eqref{phi_nn} and \eqref{ryadok}, that
\begin{multline}\label{ex2}\widehat{\phi}_{t_1,\ldots,t_n;\boldsymbol{k}}(\omega)=\frac{\beta_n\ \widehat{B}_{n,\boldsymbol{k}}(\omega)\,
\mathrm{e}^{- c_{1/r}i\omega}}{\bigl[\prod_{j\in J_r} (1+\mathrm{e}^{i\omega}r_j(n))\bigr]\bigl[\prod_{\iota\in J_{1/r}} (1+\mathrm{e}^{-i\omega}r_\iota(n))\bigr]}\\=
\beta_n \prod_{j\in J_r}\sum_{l_j=0}^\infty \bigl(-r_j(n)\,\mathrm{e}^{i\omega}\bigr)^{l_j}\
\prod_{\iota\in J_{1/r}}\sum_{l_\iota=0}^\infty (-r_\iota(n)\,\mathrm{e}^{-i\omega})^{l_\iota}\,\widehat{B}_{n,\boldsymbol{k}}(\omega)\ \mathrm{e}^{-c_{1/r}i\omega}.\end{multline}

The Fourier transform of a wavelet ${\psi}_{t_1,\ldots,t_n;\boldsymbol{k},\boldsymbol{s}}$, $\boldsymbol{s}\in\mathbb{Z}$, related to ${\phi}_{t_1,\ldots,t_n;\boldsymbol{k}}$ has the form
\begin{multline}\label{Npsi_n}\widehat{\psi}_{t_1,\ldots,t_n;\boldsymbol{k},\boldsymbol{s}}(\omega)=\frac{\sqrt{\alpha_1(n)\,t_1(n)\ldots\alpha_n(n)\,t_n(n)}\,\mathrm{e}^{-i\omega\boldsymbol{s}}}{2\,\mathrm{e}^{i\omega/2}\,\mathrm{e}^{i\pi(n+1+\boldsymbol{k})}}\\ \times\frac{\Bigl[\bigl(1-\mathrm{e}^{-i\omega/2}t_1(n)\bigr)\ldots \bigl(1-\mathrm{e}^{-i\omega/2}t_n(n)\bigr)\Bigr]\,\bigl({\mathrm{e}^{i\omega/2}-1}\bigr)^{n+1}\,\widehat{B}_{n}(\omega/2)}{\bigl(1+\mathrm{e}^{-i\omega}t_1(n)\bigr)\bigl(1+\mathrm{e}^{i\omega/2}t_1(n)\bigr)\ldots \bigl(1+\mathrm{e}^{-i\omega}t_n(n)\bigr)
\bigl(1+\mathrm{e}^{i\omega/2}t_n(n)\bigr)}\end{multline} (see \cite[\S\,2]{JMAA} for the case $t_j(n)=r_j(n)$ for all $j=1,\ldots,n$ or \cite[\S\,3.2]{RMC} for general situation), where
$$\mathrm{e}^{-i\omega/2}\bigl(\mathrm{e}^{i\omega/2}-1\bigr)^{n+1}=\sum_{k=0}^{n+1}\frac{(-1)^k(n+1)!}{k!(n+1-k)!}\mathrm{e}^{(n-k)i\omega/2}$$ together with \eqref{diff'} 
ensures 
fulfilment of the cancellation property  \begin{equation}\label{zero}\int_\mathbb{R} x^m \,{\psi}_{t_1,\ldots,t_n}(x)\,dx=0,\qquad(m=0,1,\ldots,n).\end{equation} Similarly to the situation with $\boldsymbol{k}$, the parameter $\boldsymbol{s}\in\mathbb{Z}$ in \eqref{Npsi_n} is also fixed. It makes possible the positioning  ${\psi}_{t_1,\ldots,t_n;\boldsymbol{k},\boldsymbol{s}}$ at a particular point on $\mathbb{R}$ (see e.g. \cite[p. 25]{RMC}).

Multiplying the numerator and denominator in \eqref{Npsi_n} by $\bigl(1-\mathrm{e}^{i\omega/2}t_1(n)\bigr)\ldots \bigl(1-\mathrm{e}^{i\omega/2}t_n(n)\bigr)$ we obtain
\begin{multline}\label{Npsi_nL}\widehat{\psi}_{t_1,\ldots,t_n;\boldsymbol{k},\boldsymbol{s}}(\omega)=\frac{\sqrt{\alpha_1(n)\,t_1(n)\ldots\alpha_n(n)\,t_n(n)}\,\mathrm{e}^{-i\omega\boldsymbol{s}}}{2\,\mathrm{e}^{i\omega/2}\,\mathrm{e}^{i\pi(n+1+\boldsymbol{k})}}\\\times
\frac{\Bigl[\bigl|1-\mathrm{e}^{i\omega/2}t_1(n)\bigr|^2\ldots \bigl|1-\mathrm{e}^{i\omega/2}t_n(n)\bigr|^2\Bigr]\bigl({\mathrm{e}^{i\omega/2}-1}\bigr)^{n+1}}{\bigl(1+\mathrm{e}^{-i\omega}t_1(n)\bigr)\bigl(1-\mathrm{e}^{i\omega}t_1^2(n)\bigr)\ldots \bigl(1+\mathrm{e}^{-i\omega}t_n(n)\bigr)
\bigl(1-\mathrm{e}^{i\omega}t_n^2(n)\bigr)}\ \widehat{B}_{n}(\omega/2).\end{multline} Denote $\rho_j:=\rho_j(n)=r_j(n)+1/r_j(n)$, $j=1,\ldots,n$.
Since
\begin{equation}\label{modul}
[1-\mathrm{e}^{- i\omega/2}t]\,[1-\mathrm{e}^{i\omega/2}t]=|1-\mathrm{e}^{i\omega/2}t|^2=t\Bigl[\bigl(t+1/t\bigr)-\bigl(\mathrm{e}^{i\omega/2}+\mathrm{e}^{-i\omega/2}\bigr)\Bigr]\qquad(t>0)
\end{equation} then \begin{equation*}
\bigl|1-\mathrm{e}^{i\omega/2}t_1(n)\bigr|^2\ldots \bigl|1-\mathrm{e}^{i\omega/2}t_n(n)\bigr|^2=\bigl[t_1(n)\ldots t_n(n)\bigr]\sum_{j=0}^n(-1)^j\lambda_j\cdot\cos(j\omega/2),\end{equation*} where $\lambda_n=2$, $0<\lambda_j=\lambda_j(\rho_1,\ldots,\rho_n)$ for $j\not=n$ and $\lambda_j$ is even for all $j\not=0$. Taking {into} account that
\begin{equation}\label{e-perehod}\frac{1}{[1+\mathrm{e}^{- i\omega}/r] [1+\mathrm{e}^{i\omega/2}/r]}=
\frac{r^2\,\mathrm{e}^{i\omega/2}}{[1+\mathrm{e}^{i\omega}r] [1+\mathrm{e}^{-i\omega/2}r]}\qquad(r>0),\end{equation} we obtain for fixed $\boldsymbol{k},\boldsymbol{s}\in\mathbb{Z}$, in view of \eqref{ryadok}: \begin{multline}\label{Npsihat}\widehat{\psi}_{t_1,\ldots,t_n;\boldsymbol{k},\boldsymbol{s}}(w)=\frac{t_1(n)\ldots t_n(n)\sqrt{\alpha_1(n)t_1(n)\ldots\alpha_n(n)t_n(n)}}{2\cdot(-1)^{n+1+\boldsymbol{k}}\cdot\mathrm{e}^{ i\omega\boldsymbol{s}}}\\\times\biggl\{\sum_{j=0}^n\frac{\lambda_j}{2(-1)^j}\bigl[\mathrm{e}^{ji\omega/2}+
\mathrm{e}^{-ji\omega/2}\bigr]\biggr\}\sum_{k=0}^{n+1}\frac{(-1)^k(n+1)!}{k!(n+1-k)!}\mathrm{e}^{(n-k)i\omega/2}\\
\times \prod_{j\in J_r}\sum_{m_j=0}^\infty \bigl(-r_j(n)\,\mathrm{e}^{- i\omega}\bigr)^{m_j}\sum_{l_j=0}^\infty \bigl(r_j^2(n)\,\mathrm{e}^{ i\omega}\bigr)^{l_j}\
\prod_{\iota\in J_{1/r}}\frac{r_\iota^3(n)}{(-1)^{c_{1/r}}}\sum_{m_\iota= 0}^\infty \bigl(-r_\iota(n)\,\mathrm{e}^{ i\omega}\bigr)^{m_\iota}\sum_{l_\iota= 0}^\infty \bigl(r_\iota^2(n)\,\mathrm{e}^{- i\omega}\bigr)^{l_\iota}\ \widehat{B}_{n}(\omega/2) \\
=\frac{r_1(n)\ldots r_n(n)\sqrt{\alpha_1(n)r_1(n)\ldots\alpha_n(n)r_n(n)}}{2\cdot(-1)^{n+1+\boldsymbol{k}+c_{1/r}}\cdot\mathrm{e}^{i\omega\boldsymbol{s}}}\biggl\{\sum_{j=0}^n\frac{\lambda_j}{2(-1)^j}\bigl[\mathrm{e}^{ji\omega/2}+
\mathrm{e}^{-ji\omega/2}\bigr]\biggr\}\sum_{k=0}^{n+1}\frac{(-1)^k(n+1)!}{k!(n+1-k)!}\mathrm{e}^{(n-k)i\omega/2}\\
\times \prod_{j\in J_r}\sum_{m_j=0}^\infty \bigl(-r_j(n)\,\mathrm{e}^{- i\omega}\bigr)^{m_j}\sum_{l_j=0}^\infty \bigl(r_j^2(n)\,\mathrm{e}^{ i\omega}\bigr)^{l_j}\
\prod_{\iota\in J_{1/r}}\sum_{m_\iota= 0}^\infty \bigl(-r_\iota(n)\,\mathrm{e}^{ i\omega}\bigr)^{m_\iota}\sum_{l_\iota= 0}^\infty \bigl(r_\iota^2(n)\,\mathrm{e}^{- i\omega}\bigr)^{l_\iota}\ \widehat{B}_{n}(\omega/2).\end{multline}
Orthonormal spline wavelet systems $\{
{\phi}_{t_1,\ldots,t_n;\boldsymbol{k}}, {\psi}_{t_1,\ldots,t_n;\boldsymbol{k},\boldsymbol{s}}\}=:\{
{\phi}_{n,\boldsymbol{k}}, {\psi}_{n,\boldsymbol{k},\boldsymbol{s}}\}$ are from multiresolution analysis generated by $B_n$ for any $\boldsymbol{k},\boldsymbol{s}\in\mathbb{Z}$ and for any choice of $t_i(n)\in\{r_i(n),1/r_i(n)\}$, $i=1,\ldots,n$, in either ${\phi}_{t_1,\ldots,t_n;\boldsymbol{k}}$ or ${\psi}_{t_1,\ldots,t_n;\boldsymbol{k},\boldsymbol{s}}$, {independent of each other} \cite[\S\,3.2]{RMC}.
Substitution $x=\tilde{x}-1/2$ into the definition of $B_n$ leads to  another type of Battle--Lemari\'{e} wavelet systems of natural orders $\{
{\phi}_{n,\tilde{\boldsymbol{k}}}, {\psi}_{n,{\boldsymbol{k}},\tilde{\boldsymbol{s}}}\}$ with $\tilde{\boldsymbol{k}}=\boldsymbol{k}+1/2$ and $\tilde{\boldsymbol{s}}=\boldsymbol{s}+1/2$, which coincide with $\{
{\phi}_{n,\boldsymbol{k}}, {\psi}_{n,\boldsymbol{k},\boldsymbol{s}}\}$ shifted in $1/2$ to the right. These are from the multiresolution analysis generated by $\tilde{B}_n(x):=B_n(x-1/2)$.

\smallskip
Fix $n$ and choose $t_1(n),\ldots,t_n(n)$. Put ${\beta}_n:=2^n\sqrt{\alpha_1(n)r_1(n)\ldots\alpha_n(n)r_n(n)}$, ${\gamma}_n:=\bigl[r_1(n)\ldots r_n(n)\bigr]$ $\sqrt{\alpha_1(n)r_1(n)\ldots\alpha_n(n)r_n(n)}$ and $\mathscr{A}_n(\omega):=\mathbf{A}_n(\omega/2)\mathcal{A}_n(-\omega/2)=\bigl(1-\mathrm{e}^{i\omega}r_1^2(n)\bigr)\ldots\bigl
(1-\mathrm{e}^{i\omega}r_n^2(n)\bigr)$, where $$\mathbf{A}_n(\omega):=\bigl(1+\mathrm{e}^{i\omega}r_1(n)\bigr)\ldots\bigl
(1+\mathrm{e}^{i\omega}r_n(n)\bigr), \quad \mathcal{A}_n(\omega):=\overline{\mathbf{A}_n(\omega+\pi)}=\bigl(1-\mathrm{e}^{-i\omega}r_1(n)\bigr)\ldots\bigl
(1-\mathrm{e}^{-i\omega}r_n(n)\bigr).$$

{The functions ${\phi}_{t_1,\ldots,t_n;\boldsymbol{k}}=:{\phi}_{n,\boldsymbol{k}}$ and ${\psi}_{t_1,\ldots,t_n;\boldsymbol{k},\boldsymbol{s}}=:{\psi}_{n,\boldsymbol{k},\boldsymbol{s}}$ with
\begin{equation*}\label{phi_nn'}
\hat{\phi}_{n,\boldsymbol{k}}(\omega)=\beta_n\,\frac{\mathrm{e}^{-i\omega(\boldsymbol{k}+c_{1/r})}\,\hat{B}_{n}(\omega)}{\mathbf{A}_n(\omega)}\end{equation*} and
\begin{equation*}\hat{\psi}_{n,\boldsymbol{k},\boldsymbol{s}}(\omega)=\frac{\gamma_{n}(-1)^{n+1+\boldsymbol{k}+c_{1/r}}}{2 \,\mathrm{e}^{i\omega\boldsymbol{s}}}\cdot
\frac{\bigl|\mathcal{A}_n(\omega/2)\bigr|^2\displaystyle\sum_{k=0}^{n+1}\frac{(-1)^k(n+1)!}{k!(n+1-k)!}\,\mathrm{e}^{(n-k)i\omega/2}\,\hat{B}_{n}(\omega/2)}{\mathbf{A}_n(-\omega)\mathscr{A}_n(\omega)}\end{equation*}
have unbounded} {supports on $\mathbb{R}$ (see \eqref{ex2} and \eqref{Npsihat}). In what follows we shall operate with their localised versions instead (see e.g. \cite[\S\,3.2]{RMC*}):}
\begin{equation}\label{MMM0}\widehat{\mathbf{\Phi}}_{t_1,\ldots,t_n;\boldsymbol{k}}(\omega)=\widehat{\phi}_{t_1,\ldots, t_n;\boldsymbol{k}}(\omega)\mathbf{A}_n(\omega)=\beta_n\,\widehat{B}_{n,\boldsymbol{k}+c_{1/r}}(\omega)\end{equation} and 
\begin{multline}\label{MMM}\widehat{\mathbf{\Psi}}_{t_1,\ldots,t_n;\boldsymbol{m}(\Bbbk);\boldsymbol{k},\boldsymbol{s}}(\omega)=
\widehat{\psi}_{n,\boldsymbol{k},\boldsymbol{s}}(\omega){\mathbf{A}_n(-\omega)\mathscr{A}_n(\omega)}\,\bigl|\mathscr{A}_{\boldsymbol{m}}(\omega)\bigr|^{2\Bbbk}\\=
\frac{\gamma_{n}(-1)^{n+1+\boldsymbol{k}+c_{1/r}}}{2\,\mathrm{e}^{i\omega\boldsymbol{s}}}\,\bigl|\mathcal{A}_n(\omega/2)\bigr|^2\,\bigl|\mathscr{A}_{\boldsymbol{m}}(\omega)\bigr|^{2\Bbbk}\sum_{k=0}^{n+1}\frac{(-1)^k(n+1)!}{k!(n+1-k)!}
\mathrm{e}^{(n-k)i\omega/2}\,\widehat{B}_{n}(\omega/2).\end{multline} Here $\boldsymbol{m}(\Bbbk)=\boldsymbol{m}\in\mathbb{N}$ is some fixed number and $\Bbbk\in\{0,1\}$. If $\Bbbk=0$ then \begin{equation}\label{dop1}\widehat{\mathbf{\Psi}}_{t_1,\ldots,t_n;\boldsymbol{m}(0)=\emptyset;\boldsymbol{k},\boldsymbol{s}}=:\widehat{\mathbf{\Psi}}_{t_1,\ldots,t_n;\boldsymbol{k},\boldsymbol{s}}=(-1)^{c_{1/r}}\widehat{\psi}_{r_1,\ldots,r_n;\boldsymbol{k},\boldsymbol{s}}(\omega){\mathbf{A}_n(-\omega)\mathscr{A}_n(\omega)}\end{equation} (see \cite[\S~3.2.2]{RMC}). It was established in \cite[\S\,3.2]{RMC*} for the case when $t_j(n)=r_j(n)$, $j=1,\ldots,n$, that ${\mathbf{\Phi}}_{t_1,\ldots,t_n;\boldsymbol{k}}$ and ${\mathbf{\Psi}}_{t_1,\ldots,t_n;\boldsymbol{m}(\Bbbk);\boldsymbol{k},\boldsymbol{s}}$ are finite number linear combinations of integer shifts of ${{\phi}}_{n,\boldsymbol{k}}$ and ${{\psi}}_{n,\boldsymbol{k},\boldsymbol{s}}$, respectively. Algorithms for their constructing can be seen in \cite[\S\,3.2]{RMC*} or \cite{RMC}. The same is true for any choice of $t_j(n)=r_j^\pm(n)$, $j=1,\ldots,n$. 
It follows from \eqref{MMM0} and \eqref{MMM}{, respectively,} that
\begin{equation*}\label{dr}\mathbf{\Phi}_{t_1,\ldots,t_n;\boldsymbol{k}}(x)={\beta}_{n}\,B_{n,\boldsymbol{k}+c_{1/r}}(x),\end{equation*}
\begin{equation*}
\mathbf{\Psi}_{t_1,\ldots,t_n,;\boldsymbol{k},\boldsymbol{s}}(x)=\frac{\gamma_{n}(-1)^{n+1+\boldsymbol{k}+c_{1/r}}}{2}
\Bigl[
\sum_{j=0}^{n}\frac{\lambda_{j}}{2(-1)^{j}}
\bigl[B_{2n+1}^{(n+1)}\bigl(2(x-\boldsymbol{s})+n+j\bigr)+B_{2n+1}^{(n+1)}\bigl(2(x-\boldsymbol{s})+n-j\bigr)\Bigr],
\end{equation*} where $\lambda_n=2$, $0<\lambda_j=\lambda_j(r_1+1/r_1,\ldots,r_n+1/r_n)$ for $j\not=n$ and $\lambda_j$ is even for all $j\not=0$ (see \cite[\S\,3.1]{RMC*} for detail). 

It was established in \cite[\S\,3.2]{RMC*} that $\mathbf{\Phi}_{r_1,\ldots,r_n;\boldsymbol{k}}=\sum_{\kappa=0}^n\boldsymbol{\alpha}'_\kappa\cdot \phi_{r_1,\ldots,r_n;\boldsymbol{k}-\kappa}$, where
\begin{equation}\label{positive0}\sum_{\kappa=0}^n\boldsymbol{\alpha}'_\kappa=\bigl(1+r_1(n)\bigr)\ldots\bigl(1+r_n(n)\bigr)=:\mathbf{\Lambda}'_n>0,\end{equation}
and $\mathbf{\Psi}_{r_1,\ldots,r_n;\boldsymbol{m}(\Bbbk);\boldsymbol{k},\boldsymbol{s}}=\sum_{|\kappa|\le n+\boldsymbol{m}\Bbbk}\boldsymbol{\alpha}''_\kappa\cdot \psi_{r_1\ldots,r_n;\boldsymbol{k},\boldsymbol{s}+\kappa}$, where
\begin{multline}\label{positive}\sum_{|\kappa|\le n+\boldsymbol{m}\Bbbk}\boldsymbol{\alpha}''_\kappa=\Bigl[\bigl(1+r_1(n)\bigr)\bigl(1-r_1^2(n)\bigr)\ldots\bigl(1+r_n(n)\bigr)\bigl(1-r_n^2(n)\bigr)\Bigr]\\\times\Bigl[\bigl(1-r_1^2(\boldsymbol{m})\bigr)\ldots\bigl(1-r_m^2(\boldsymbol{m})\bigr)\Bigr]^{2\Bbbk}=:\mathbf{\Lambda}''_n>0\end{multline} (see \eqref{MMM0}, \eqref{MMM} and definitions of $\mathbf{A}_n(\omega)$ and $\mathscr{A}_n(\omega)$). The same is true for any choice of $t_j(n)$, $j=1,\ldots,n$. 
It holds that \begin{equation}\label{vazhno}\textrm{supp}\,\mathbf{\Phi}_{r_1,\ldots,r_n;\boldsymbol{k}}=[\boldsymbol{k},\boldsymbol{k}+n+1]\quad\textrm{and}\quad\textrm{supp}\,\mathbf{\Psi}_{r_1,\ldots,r_n;\boldsymbol{m}(\Bbbk);\boldsymbol{k},\boldsymbol{s}}=[\boldsymbol{s}-n/2-\boldsymbol{m}\Bbbk,\boldsymbol{s}+3n/2+\boldsymbol{m}\Bbbk+1].\end{equation} The functions $\mathbf{\Phi}_{t_1,\ldots,t_n;\boldsymbol{k}}$ and $\mathbf{\Psi}_{t_1,\ldots,t_n;\boldsymbol{m}(\Bbbk);\boldsymbol{k},\boldsymbol{s}}$ are finite linear combinations of integer shifts of $\phi_{t_1,\ldots,t_n;\boldsymbol{k}}$ and $\psi_{t_1,\ldots,t_n;\boldsymbol{k},\boldsymbol{s}}$, respectively, which are elements of the same orthonormal basis in $\textrm{MRA}_{B_n}$ of $L_2(\mathbb{R})$.
The system $\{\mathbf{\Phi}_{n,\boldsymbol{k}},\mathbf{\Psi}_{n,\boldsymbol{m},\boldsymbol{k},\boldsymbol{s}}\}$ forms a semi--orthogonal Riesz basis in $L_2(\mathbb{R})$ independently of the choice of $t_j(n)$, $j=1,\ldots,n$, in either $\mathbf{\Phi}_{n,\boldsymbol{k}}:=\mathbf{\Phi}_{t_1,\ldots,t_n;\boldsymbol{k}}$ or $\mathbf{\Psi}_{n,\boldsymbol{m},\boldsymbol{k},\boldsymbol{s}}:=\mathbf{\Psi}_{t_1,\ldots,t_n;m(\Bbbk);\boldsymbol{k},\boldsymbol{s}}$ (see \cite[\S\,3]{RMC*} for details).

Observe that localised systems $\{\mathbf{\Phi}_{n,\tilde{\boldsymbol{k}}},\mathbf{\Psi}_{n,\boldsymbol{m},\boldsymbol{k},\tilde{\boldsymbol{s}}}\}$ related to $\{{\phi}_{n,\tilde{\boldsymbol{k}}},{\psi}_{n,\boldsymbol{k},\tilde{\boldsymbol{s}}}\}$ with $\tilde{\boldsymbol{k}}=\boldsymbol{k}+1/2$ and $\tilde{\boldsymbol{s}}=\boldsymbol{s}+1/2$ are coinciding with $\{\mathbf{\Phi}_{n,\boldsymbol{k}},\mathbf{\Psi}_{n,\boldsymbol{m},\boldsymbol{k},\boldsymbol{s}}\}$ shifted in $1/2$ to the right.

\section{Atomic and spline wavelet decompositions in $A_{pq}^{s,w}(\mathbb{R}^N)$ with $w\in\mathscr{A}_\infty^\loc$} \label{Decom}

Let $C(\mathbb{R}^N)$ be the space of all complex--valued uniformly continuous bounded functions in $\mathbb{R}^N$ and let $$C^M(\mathbb{R}^N)=\bigl\{f\in C(\mathbb{R}^N)\colon D^\gamma f\in C(\mathbb{R}^N),\,|\gamma|\le M\bigr\},\qquad M\in\mathbb{N}_0,$$ obviously normed. We shall use the convention $C^0(\mathbb{R}^N)=C(\mathbb{R}^N)$.


\smallskip
For $l=1,\ldots,N$ we fix $n_l,\boldsymbol{m}_l\in\mathbb{N}$, $\Bbbk_l\in\{0,1\}$ and $\boldsymbol{k}_l$, $\boldsymbol{s}_l\in\mathbb{Z}$.
Let $\boldsymbol{\kappa}_l\in\{\boldsymbol{k}_l,\tilde{\boldsymbol{k}_l}=\boldsymbol{k}_l+1/2\}$, $\boldsymbol{\varkappa}_l\in\{\boldsymbol{s}_l,\tilde{\boldsymbol{s}}_l=\boldsymbol{s}_l+1/2\}$. {For each $l\in\{1,\ldots,N\}$ we make a choice of $t_1(n_l),\ldots,t_{n_l}(n_l)$ and denote
\begin{equation}\label{vazhnoo}\widetilde{\mathbf{\Phi}}_{n_l,\boldsymbol{\kappa}_l}(x):=(\mathbf{\Lambda}'_{n_l})^{-1}\mathbf{\Phi}_{n_l,{\boldsymbol{\kappa}_l}}(x)\quad\textrm{and}\quad \widetilde{\mathbf{\Psi}}_{n_l,\boldsymbol{m}_l(\Bbbk_l),\boldsymbol{k}_l,\boldsymbol{\varkappa}_l}(x):=(\mathbf{\Lambda}^{''}_{n_l})^{-1}(-1)^{\boldsymbol{k}_l+c_{1/r}}\mathbf{\Psi}_{n_l,\boldsymbol{m}_l(\Bbbk_1),\boldsymbol{k}_l,\boldsymbol{\varkappa}_l}(x).\end{equation} Here $\boldsymbol{\kappa}_l$ and $\boldsymbol{\varkappa}_l$ must be both either integer or non--integer. For every $l\in\{1,\ldots,N\}$ we have $\widetilde{\mathbf{\Phi}}_{n_l,\boldsymbol{\kappa}_l}\in C^{n_l-1}(\mathbb{R})$, $\widetilde{\mathbf{\Psi}}_{n_l,\boldsymbol{m}_l(\Bbbk_l),\boldsymbol{k}_l,\boldsymbol{\varkappa}_l}\in C^{n_l-1}(\mathbb{R})$, and the system $\{\widetilde{\mathbf{\Phi}}_{n_l,\boldsymbol{\kappa}_l},\widetilde{\mathbf{\Psi}}_{n_l,\boldsymbol{m}_l,\boldsymbol{k}_l,\boldsymbol{\varkappa}_l}\}$ forms a Riesz basis in $L_2(\mathbb{R})$ within $\mathrm{MRA}_{B_{n_l,\boldsymbol{k}_l}}$ (if $\boldsymbol{\kappa}_l=\boldsymbol{k}_l$) or $\mathrm{MRA}_{\tilde{B}_{n_l,\boldsymbol{k}_l}}=:\mathrm{MRA}_{{B}_{n_l,\tilde{\boldsymbol{k}}_l}}$ (if $\boldsymbol{\kappa}_l=\tilde{\boldsymbol{k}}_l$).

Let $\mathscr{V}_{d;\boldsymbol{\kappa}_l}$, $d\in\mathbb{Z}$, be the multiresolution approximation of $L_2(\mathbb{R})$ generated by $B_{n_l,\boldsymbol{\kappa}_l}$, and $\mathscr{V}_{d+1;\boldsymbol{\kappa}_l}=\mathscr{V}_{d;\boldsymbol{\kappa}_l}\oplus\mathscr{W}_{d;\boldsymbol{\varkappa}_l}$ for each $d\in\mathbb{N}_0$. The usual tensor--product procedure yields the following compactly supported scaling function and associated wavelets on $\mathbb{R}^N$ related to $\{\widetilde{\mathbf{\Phi}}_{n_l,\boldsymbol{\kappa}_l},\widetilde{\mathbf{\Psi}}_{n_l,\boldsymbol{m}_l,\boldsymbol{k}_l,\boldsymbol{\varkappa}_l}\}$, $l=1,\ldots,N$:
\begin{equation}\label{wavelets_main}{\mathbf{\Phi}}\in C^{n_0-1}(\mathbb{R}^N),\qquad {\mathbf{\Psi}}_{i}\in C^{n_0-1}(\mathbb{R}^N)\quad (i=1,\ldots,2^N-1);\qquad
n_0=\min\{n_1,\ldots,n_N\}.\end{equation}
By $V_{d;\boldsymbol{\kappa}_1,\ldots,\boldsymbol{\kappa}_N}$, $d\in\mathbb{N}_0$, we denote the closure in $L_2(\mathbb{R}^N)$ norm of the tensor product
${\mathscr{V}_{d;\boldsymbol{\kappa}_1}\otimes\ldots\otimes\mathscr{V}_{d;\boldsymbol{\kappa}_N}}$, $d\in\mathbb{N}_0$. Since for each $l=1,\ldots,N$, with chosen $\boldsymbol{\kappa}_l$ and $\boldsymbol{\varkappa}_l$, the system $\{\widetilde{\mathbf{\Phi}}_{n_l,\boldsymbol{\kappa}_l},\widetilde{\mathbf{\Psi}}_{n_l,\boldsymbol{m}_l,\boldsymbol{k}_l,\boldsymbol{\varkappa}_l}\}$ {forms} a Riesz basis of $L_2(\mathbb{R})$, then $\mathbf{\Phi}$ and $\mathbf{\Psi}_i$, $i=1,\ldots,2^N-1$, {forms} a Riesz basis in $V_{0;\boldsymbol{\kappa}_1,\ldots,\boldsymbol{\kappa}_N}$ and $W_{0;\boldsymbol{\varkappa}_1,\ldots,\boldsymbol{\varkappa}_N}$, respectively, where $V_{d+1;\boldsymbol{\kappa}_1,\ldots,\boldsymbol{\kappa}_N}=V_{d;\boldsymbol{\kappa}_1,\ldots,\boldsymbol{\kappa}_N}\oplus W_{d;\boldsymbol{\varkappa}_1,\ldots,\boldsymbol{\varkappa}_N}$, $d\in\mathbb{N}_0$. 
For $x\in\mathbb{R}^N$ we put \begin{equation}\label{ForRepr'}{\mathbf{\Phi}}_{\tau}(x):=\mathbf{\Phi}(x-\tau)\quad\textrm{and}\quad {\mathbf{\Psi}}_{id\tau}(x):=2^{dN/2}\mathbf{\Psi}_i(2^dx-\tau) \qquad (i=1,\ldots,2^N-1;\,\tau\in\mathbb{Z}^N;\, d\in\mathbb{N}_0),\end{equation} where $\mathbb{Z}^N$ is the set of all lattice points in $\mathbb{R}^N$ having integer components.

For $\tau\in\mathbb{Z}^N$ and $d\in\mathbb{N}_0$, let $Q_{d\tau}$ denote the $N-$dimensional cube with sides parallel to the axes of coordinates, centered at $2^{-d}\tau$ and with side length $2^{-d}$. For $0<p<\infty$, $d\in\mathbb{N}_0$ and $\tau\in\mathbb{Z}^N$ we denote by $\chi_{d\tau}^{(p)}$ the $p-$normalized characteristic function of the 
$Q_{d\tau}$: \begin{equation}\label{chi}\chi_{d\tau}^{(p)}(x):=2^{dN/p}
\chi_{d\tau}(x):=\begin{cases}2^{dN/p}, & x\in Q_{d\tau},\\ 0, & x\not\in Q_{d\tau},\end{cases}\qquad \|\chi_{d\tau}^{(p)}\|_{L^p(\mathbb{R}^N)}=1.\end{equation}

Characterisation of $A_{pq}^{s,w}(\mathbb{R}^N)$ by spline wavelets was performed in \cite[\S~4.3]{RMC*}. To remind this result, we define for $s\in\mathbb{R}$, $0<p<\infty$, $0<q\le\infty$ and $w\in\mathscr{A}_\infty^\loc$ two sequence spaces: 
\begin{multline*}b_{pq}^{s,w}:=\Biggl\{\lambda=\{\lambda_{00\tau}\}_{\tau\in\mathbb{Z}^N}\cup\{\lambda_{id\tau}\}_{{i=1,\ldots,2^N-1};\,{d\in\mathbb{N};\,\tau\in\mathbb{Z}^N}}\colon\,\lambda_{id\tau}\in\mathbb{C},\\ \|\lambda\|_{{b}^{s,w}_{pq}}=\biggl(\int_{\mathbb{R}^N}\Bigl(\sum_{\tau\in\mathbb{Z}^N}\bigl|\lambda_{00\tau}\bigr|\chi_{0\tau}^{(p)}(x)\Bigr)^pw(x)\,dx\biggr)^{\frac{1}{p}}\\+
\Biggl(\sum_{d=1}^\infty 2^{d(s-N/p)q}\sum_{i=1}^{2^N-1}\biggl(\int_{\mathbb{R}^N}\Bigl(\sum_{\tau\in\mathbb{Z}^N} \bigl|\lambda_{id\tau}\bigr|\chi_{d\tau}^{(p)}(x)\Bigr)^p w(x)\,dx\biggr)^\frac{q}{p}\Biggr)^\frac{1}{q}<\infty\Biggr\},\end{multline*}
\begin{multline*}f_{pq}^{s,w}:=\Biggl\{\lambda=\{\lambda_{00\tau}\}_{\tau\in\mathbb{Z}^N}\cup\{\lambda_{id\tau}\}_{{i=1,\ldots,2^N-1};\,{d\in\mathbb{N};\,\tau\in\mathbb{Z}^N}}\colon\,\lambda_{id\tau}\in\mathbb{C},\\ \|\lambda\|_{{f}^{s,w}_{pq}}=
\biggl(\int_{\mathbb{R}^N}\Bigl(\sum_{\tau\in\mathbb{Z}^N}\Bigr| \lambda_{00\tau}\chi_{0\tau}^{(p)}(x)\Bigr|^q\biggr)^\frac{p}{q}w(x)\,dx\biggr)^{\frac{1}{p}}\\
+\Biggl(\int_{\mathbb{R}^N}\biggl(\sum_{d=1}^\infty 2^{dsq}\sum_{i=1}^{2^N-1}\sum_{\tau\in\mathbb{Z}^N}\Bigr| \lambda_{id\tau}\chi_{d\tau}^{(p)}(x)\Bigr|^q\biggr)^\frac{p}{q}w(x)\,dx\Biggr)^{\frac{1}{p}}<\infty\Biggr\}.\end{multline*} To simplify the notations we write ${a}_{pq}^{s,w}$ instead of ${b}_{pq}^{s,w}$ and ${f}^{s,w}_{pq}$.
For $w\in\mathscr{A}_\infty^\loc$ with ${r}_0$ of the form \eqref{r_w} $$\sigma_{p}(w):=N\Bigl(\frac{{r}_0}{\min\{p,{r}_0\}}-1\Bigr)+N({r}_0-1),\quad \sigma_q:=\frac{N}{\min\{1,q\}}-N,\quad \sigma_{pq}(w):=\max\bigl\{\sigma_p(w),\sigma_q\bigr).$$
\begin{theorem}\label{main'}
Let $0<p<\infty$, $0<q\le\infty$, $-\infty<s<+\infty$ and $w\in\mathscr{A}_\infty^\loc$. Let 
${\mathbf{\Phi}},\, {\mathbf{\Psi}}_{i}\in C^{n_0-1}(\mathbb{R}^N)$ with $i=1,\ldots,2^N-1$ be functions satisfying \eqref{wavelets_main}. We assume \begin{equation}\label{condB}n_0\ge \max\Bigl\{0,[s]+1,{\bigl[N({r}_0-1)/p-s\bigr]}+1,\bigl[\sigma_{p}(w)-s\bigr]\Bigr\}+1\end{equation} in the case $A_{pq}^{s,w}(\mathbb{R}^N)=B_{pq}^{s,w}(\mathbb{R}^N)$, and
\begin{equation}\label{condF}n_0\ge \max\Bigl\{0,[s]+1,{\bigl[N({r}_0-1)/p-s\bigr]}+1,\bigl[\sigma_{pq}(w)-s\bigr]\Bigr\}+1\end{equation} when $A_{pq}^{s,w}(\mathbb{R}^N)$ denotes $F_{pq}^{s,w}(\mathbb{R}^N)$. Then $f\in\mathscr{S}'_e(\mathbb{R}^N)$ belongs to $A_{pq}^{s,w}(\mathbb{R}^N)$ if and only if it can be represented as
\begin{equation*}
f=\sum_{\tau\in\mathbb{Z}^N} \lambda_{00\tau}\mathbf{\Phi}_\tau+
\sum_{d\in\mathbb{N}}\sum_{i=1}^{2^N-1}\sum_{\tau\in\mathbb{Z}^N}\lambda_{id\tau}2^{-dN/2}\mathbf{\Psi}_{i(d-1)\tau},
\end{equation*} where $\lambda\in a_{pq}^{s,w}$ and the series converges in $\mathscr{S}'_e(\mathbb{R}^N)$. This representation is unique with
\begin{equation}\label{asty}\lambda_{00\tau}=\langle f,\mathbf{\Phi}_{\tau}\rangle\quad(\tau\in\mathbb{Z}^N),\qquad   \lambda_{id\tau}=2^{dN/2}\langle f,\mathbf{\Psi}_{i(d-1)\tau}\rangle\quad(i=1,\ldots,2^N-1;\,d\in\mathbb{N};\,\tau\in\mathbb{Z}^N)\end{equation}
and $I\colon f\mapsto\Bigl\{\bigl\{\lambda_{00\tau}\bigr\}\cup\bigl\{\lambda_{id\tau}\bigr\}\Bigr\}$ is a linear isomorphism of $A_{pq}^{s,w}(\mathbb{R}^N)$ onto $a_{pq}^{s,w}$. Besides, \begin{equation}\label{isom}\|f\|_{A_{pq}^{s,w}(\mathbb{R}^N)}\approx \|{\lambda}\|_{{a}_{pq}^{s,w}}.\end{equation}\end{theorem}

\section{Images of integration operators}\label{Emb}
The following differentiation property of B--splines plays an essential role in this part of the work:
\begin{equation}\label{diff}
B'_n(\cdot)=B_{n-1}(\cdot)-B_{n-1}(\cdot-1)\quad\textrm{almost everywhere on }\mathbb{R}.\end{equation}
We shall mostly use its generalised form:
\begin{equation}\label{diff'}
B^{(k)}_{n}(\cdot)=\sum_{l=0}^k \frac{(-1)^l k!}{l!(k-l)!}B_{n-k}(\cdot-l)
,\quad\quad n\ge k\in\mathbb{N}.\end{equation} Recall that ${A}_{pq}^{s,w}(\mathbb{R}^N)$ stands for either ${B}_{pq}^{s,w}(\mathbb{R}^N)$ or ${F}_{pq}^{s,w}(\mathbb{R}^N)$ with a weight $w$ from $\mathscr{A}_\infty^\loc$. 

\smallskip 
We begin with an example for the case $N=2$ and the operator $I_{\Omega_1^+}^{1}$ (see \eqref{starEx} with $m_1=1$) of the form
\begin{equation}\label{left=1}
I_{\Omega_1^+}^{1} f({x_1},y_2):=\int_{0}^{x_1}{f({y_1},y_2)}\,dy_1,\qquad x_1\in\Omega_1^+=[0,+\infty).
\end{equation}
\begin{example}\label{Exam} Let $1<  p,q<\infty$ and $s=1$. Assume $u,w\in\mathscr{A}_\infty^\loc$ with $r_0(u)=r_0(w):=1$. Suppose $f\in L_1^\textrm{loc}(\mathbb{R}^2)$ and $f({y_1},y_2)\equiv 0$ for all $y_2\in\mathbb{R}$ if $y_1\in(-\infty,0)$. 
We shall demonstrate that, in the situation when weights $u$ and $w$ are of product type, that is $v(x_1,x_2)=v_1(x_1)v_2(x_2)$ for the both $v=u$ and $v=w$, then $I_{\Omega_1^+}^{1}f\in {B}_{pq}^{0,w}(\mathbb{R}^2)$ if $f\in {B}_{pq}^{1,u}(\mathbb{R}^2)$ provided \begin{align}\mathscr{M}_{\Omega_1^+}^1(d):=&\frac{1}{2^{d}}\sup_{\tau_1\ge 0}\biggl(\sum_{r\ge\tau_1}\int_{Q_{dr}}w_1\biggr)^{\frac{1}{p}}\biggl(\sum_{0\le r\le\tau_1} \Bigl(\int_{Q_{dr}}\bar{u}_1\Bigr)^{1-p'}\biggr)^{\frac{1}{p'}}<\infty\quad\forall d\in\mathbb{N}_0,\label{Md}\\\mathscr{N}_{\Omega_1^+}^1(d):=&\frac{1}{2^{2d}}\Biggl[\sup_{\tau_1\ge 0}\biggl(\sum_{r\ge\tau_1}(r-\tau_1+1)^{p}\int_{Q_{dr}}w_1\biggr)^{\frac{1}{p}}\biggl(\sum_{0\le r\le\tau_1} \Bigl(\int_{Q_{dr}}\tilde{u}_1\Bigr)^{1-p'}\biggr)^{\frac{1}{p'}}\nonumber\\&+\sup_{\tau_1\ge 0}\biggl(\sum_{r\ge\tau_1}\int_{Q_{dr}}w_1\biggr)^{\frac{1}{p}}\biggl(\sum_{0\le r\le\tau_1} (\tau_1-r+1)^{p'}\Bigl(\int_{Q_{dr}}\tilde{u}_1\Bigr)^{1-p'}\biggr)^{\frac{1}{p'}}\Biggr]<\infty\quad\forall d\in\mathbb{N},\label{Nd}\end{align} where $\tau_1\in\mathbb{Z}$ is the first component of $\tau=(\tau_1,\tau_2)\in\mathbb{Z}^2$ and $Q_{dr}=\Bigl[\frac{r-1/2}{2^d},\frac{r+1/2}{2^d}\Bigr]$. Here $\bar{u_1}\le u_1$ and $\tilde{u}_1\le u_1$ are one variable functions, we assume $u:=u_1w_2$. 
Moreover, it holds 
\begin{equation}\label{ineqEx}\|I_{\Omega_1^+}^{1}f\|_{{B}_{pq}^{0,w}(\mathbb{R}^2)}\lesssim C_{\Omega_1^+}^{1}\|f\|_{{B}_{pq}^{1,u}(\mathbb{R}^2)}\quad\textrm{with}\ \ C_{\Omega_1^+}^{1}:=\sup_{d\in\mathbb{N}_0}\bigl[\mathscr{M}_{\Omega_1^+}^1(d)+\mathscr{N}_{\Omega_1^+}^1({d+1})\bigr].\end{equation}

Observe that $\sigma_p(w)=0$ in our case. To prove our assertion we choose $n_1=2$ and $n_2=3$. Then
$n_0=\min\{n_1,n_2\}-2=0=s$ in the target space ${B}_{pq}^{s=0,w}(\mathbb{R}^2)$, which meets the condition \eqref{condB}.

On the strength of Theorem \ref{main'},
\begin{equation}\label{rysha}\|I_{\Omega_1^+}^{1}f\|_{B^{0,w}_{pq}(\mathbb{R}^N)}\approx\|{\lambda}\|_{{b}_{pq}^{0,w}},\end{equation} where
$$\lambda_{00\tau}=\langle I_{\Omega_1^+}^{1}f,\mathbf{\Phi}_{\tau}\rangle\quad(\tau=(\tau_1,\tau_2)\in\mathbb{Z}^2),\qquad   \lambda_{id\tau}=2^{d}\langle I_{\Omega_1^+}^{1}f,\mathbf{\Psi}_{i(d-1)\tau}\rangle\quad(i=1,2,3;\,d\in\mathbb{N};\,\tau\in\mathbb{Z}^2).$$ Here the system \eqref{ForRepr'} of functions $\mathbf{\Phi}_{\tau}$ and $\mathbf{\Psi}_{i(d-1)\tau}$ is generated by 
two one--dimentional systems of the form \eqref{vazhnoo} with $\boldsymbol{\kappa}_1=-3$, $\boldsymbol{\kappa}_2=-4$, $\boldsymbol{\varkappa}_1=-7$, $\boldsymbol{\varkappa}_2=-5$, $n_1=2$, $n_2=3$, $\Bbbk_1=1=\Bbbk_2+1$ and $\boldsymbol{m}_1(\Bbbk_1)=3$. To visualise the situation, for each of the two variables $x_l$ ($l=1,2$) we start from one variable functions ${\mathbf{\Phi}}_{n_l,\boldsymbol{k}_l}$ and ${\mathbf{\Psi}}_{n_l,\boldsymbol{m}_l(\Bbbk_l),\boldsymbol{k}_l,\boldsymbol{\varkappa}_l}$ such that (see \eqref{MMM0} and \eqref{MMM}) \begin{equation}\label{L1}\widehat{\mathbf{\Phi}}_{n_l,\boldsymbol{\kappa}_l}(\omega)=\widehat{\mathbf{\Phi}}_{r_1,\ldots,r_{n_l};\boldsymbol{\kappa}_l}(\omega)=\beta_{n_l}\,\widehat{B}_{n_l,\boldsymbol{\kappa}_l}(\omega)\end{equation} and 
\begin{align}\label{L4}\widehat{\mathbf{\Psi}}_{n_l,\boldsymbol{m}_l(\Bbbk_l),\boldsymbol{\kappa}_l,\boldsymbol{\varkappa}_l}(\omega)&=\widehat{\mathbf{\Psi}}_{r_1,\ldots,r_{n_l};m_l(\Bbbk_l),\boldsymbol{\kappa}_l,\boldsymbol{\varkappa}_l}(\omega)\nonumber\\&=
\frac{\gamma_{n_l}} {2}\,\bigl|\mathcal{A}_{n_l}(\omega/2)\bigr|^2\,\bigl|\mathscr{A}_{\boldsymbol{m}_l}(\omega)\bigr|^{2\Bbbk_l}\sum_{k=0}^{n_l+1}\frac{(-1)^k(n_l+1)!}{k!(n_l+1-k)!}
\mathrm{e}^{(n_l-k)i\omega/2}\,\widehat{B}_{n_l}(\omega/2)
\mathrm{e}^{\boldsymbol{\varkappa}_li\omega}
.\end{align} Let us simplify the notations and write $\widetilde{\mathbf{\Phi}}_{i(d-1)\tau_l}$ and $\widetilde{\mathbf{\Psi}}_{i(d-1)\tau_l}$ instead of $\widetilde{\mathbf{\Phi}}_{n_l,\boldsymbol{\kappa}_l}(2^{d-1}\cdot-\tau_l)=\bigl(\mathbf{\Lambda}'_{n_l}\bigr)^{-1}$ ${\mathbf{\Phi}}_{n_l,\boldsymbol{\kappa}_l}(2^{d-1}\cdot-\tau_l)$ and $\widetilde{\mathbf{\Psi}}_{n_l,\boldsymbol{m}_l(\Bbbk_l),\boldsymbol{\kappa}_l,\boldsymbol{\varkappa}_l}(2^{d-1}\cdot-\tau_l)=\bigl(\mathbf{\Lambda}^{''}_{n_l}\bigr)^{-1}{\mathbf{\Psi}}_{n_l,\boldsymbol{m}_l(\Bbbk_l),\boldsymbol{\kappa}_l,\boldsymbol{\varkappa}_l}(2^{d-1}\cdot-\tau_l)$, respectively, for $i\in\{1,2,3\}$. We have $\mathbf{\Phi}_{\tau}(x_1,x_2)=\mathbf{\Phi}_{00\tau}(x_1,x_2)=\widetilde{\mathbf{\Phi}}_{00\tau_1}(x_1)\widetilde{\mathbf{\Phi}}_{00\tau_2}(x_2)=\widetilde{\mathbf{\Phi}}_{\tau_1}(x_1)\widetilde{\mathbf{\Phi}}_{\tau_2}(x_2)$, further, $\mathbf{\Psi}_{1(d-1)\tau}(x_1,x_2)=\widetilde{\mathbf{\Psi}}_{1(d-1)\tau_1}(x_1)\widetilde{\mathbf{\Psi}}_{1(d-1)\tau_2}(x_2)$,\ $\mathbf{\Psi}_{2(d-1)\tau}(x_1,x_2)=\widetilde{\mathbf{\Psi}}_{2(d-1)\tau_1}(x_1)\widetilde{\mathbf{\Phi}}_{2(d-1)\tau_2}(x_2)$,\ $\mathbf{\Psi}_{3(d-1)\tau}$ $(x_1,x_2)=\widetilde{\mathbf{\Phi}}_{3(d-1)\tau_1}(x_1)\widetilde{\mathbf{\Psi}}_{3(d-1)\tau_2}(x_2)$ for $x=(x_1,x_2)\in\mathbb{R}^2$ and 
\begin{align}\label{L11}\|\lambda\|_{{b}^{s,w}_{pq}}=
&\mathbf{V}_{I_{\Omega_1^+}^{1}f}+\mathbf{W}_{I_{\Omega_1^+}^{1}f}:=\biggl(\int_{\mathbb{R}^2}\Bigl(\sum_{\tau\in\mathbb{Z}^2}\bigl|\langle I_{\Omega_1^+}^{1}f,\mathbf{\Phi}_{\tau}\rangle\bigr|\chi_{0\tau}^{(p)}(x)\Bigr)^pw(x)\,dx\biggr)^{\frac{1}{p}}\nonumber\\&+
\Biggl(\sum_{d=1}^\infty 2^{d(1-2/p)q}\sum_{i=1}^{3}\biggl(\int_{\mathbb{R}^2}\Bigl(\sum_{\tau\in\mathbb{Z}^2} \bigl|\langle I_{\Omega_1^+}^{1}f,\mathbf{\Psi}_{i(d-1)\tau}\rangle\bigr|\chi_{d\tau}^{(p)}(x)\Bigr)^p w(x)\,dx\biggr)^\frac{q}{p}\Biggr)^\frac{1}{q},\end{align} where, in view of disjointness of $\mathring Q_{d\tau}=\prod_{l=1}^2 \Bigl(\frac{\tau_l-1/2}{2^d},\frac{\tau_l+1/2}{2^d}\Bigr)$,
\begin{align*}\Bigl(\mathbf{V}_{I_{\Omega_1^+}^{1}f}\Bigr)^p&=\sum_{\tau\in\mathbb{Z}^2}\int_{Q_{0\tau}}\Bigl(\sum_{\tau\in\mathbb{Z}^2}\bigl|\langle I_{\Omega_1^+}^{1}f,\mathbf{\Phi}_{\tau}\rangle\bigr|\chi_{0\tau}^{(p)}(x)\Bigr)^pw(x)\,dx=
\sum_{\tau\in\mathbb{Z}^2}\bigl|\langle I_{\Omega_1^+}^{1}f,\mathbf{\Phi}_{\tau}\rangle\bigr|^p\int_{Q_{0\tau}}w\\&=\sum_{\tau_1\in\mathbb{Z}}
\sum_{\tau_2\in\mathbb{Z}}\bigl|\langle I_{\Omega_1^+}^{1}f,\widetilde{\mathbf{\Phi}}_{\tau_1}\widetilde{\mathbf{\Phi}}_{\tau_2}\rangle\bigr|^p\int_{Q_{0\tau_1}}w_1\,\int_{Q_{0\tau_2}}w_2=:\mathbf{V}_{I_{\Omega_1^+}^{1}f}^{(0)}
\end{align*}and, analogously, by revealing \eqref{chi},\begin{multline}\label{L2}\Bigl(\mathbf{W}_{I_{\Omega_1^+}^{1}f}\Bigr)^q=\sum_{d=1}^\infty 2^{dq}
\Biggl\{\sum_{i=1}^{3}\biggl(\int_{\mathbb{R}^2}\Bigl(\sum_{\tau\in\mathbb{Z}^2} \bigl|\langle I_{\Omega_1^+}^{1}f,\mathbf{\Psi}_{i(d-1)\tau}\rangle\bigr|\chi_{d\tau}(x)\Bigr)^p w(x)\,dx\biggr)^\frac{q}{p}\Biggr\}\\=
\sum_{d=1}^\infty 2^{dq}
\Biggl\{\biggl(\sum_{\tau_1\in\mathbb{Z}}
\sum_{\tau_2\in\mathbb{Z}}\bigl|\langle I_{\Omega_1^+}^{1}f,\widetilde{\mathbf{\Psi}}_{1(d-1)\tau_1}\widetilde{\mathbf{\Psi}}_{1(d-1)\tau_2}\rangle\bigr|^p \int_{Q_{d\tau_1}}w_1\,\int_{Q_{d\tau_2}}w_2\biggr)^\frac{q}{p}\\+\biggl(\sum_{\tau_1\in\mathbb{Z}}
\sum_{\tau_2\in\mathbb{Z}}\bigl|\langle I_{\Omega_1^+}^{1}f,\widetilde{\mathbf{\Psi}}_{2(d-1)\tau_1}\widetilde{\mathbf{\Phi}}_{2(d-1)\tau_2}\rangle\bigr|^p \int_{Q_{d\tau_1}}w_1\,\int_{Q_{d\tau_2}}w_2\biggr)^\frac{q}{p}\\+\biggl(\sum_{\tau_1\in\mathbb{Z}}
\sum_{\tau_2\in\mathbb{Z}}\bigl|\langle I_{\Omega_1^+}^{1}f,\widetilde{\mathbf{\Phi}}_{3(d-1)\tau_1}\widetilde{\mathbf{\Psi}}_{3(d-1)\tau_2}\rangle\bigr|^p \int_{Q_{d\tau_1}}w_1\,\int_{Q_{d\tau_2}}w_2\biggr)^\frac{q}{p}\Biggr\}=:\sum_{d=1}^\infty 2^{dq}\Bigl\{\sum_{i=1}^3\mathbf{W}_{I_{\Omega_1^+}^{1}f}^{i(d-1)}\Bigr\}.\end{multline} Denote $\int_{Q_{d\tau_1}}w_1=:w_{d\tau_1}$ and write for a fixed $\tau_2$, taking into account that $I_{\Omega_1^+}^{1}f=0$ on $\mathbb{R}^2\setminus\mathbb{R}^2_+$, where $\mathbb{R}_+^2:=(0,\infty)\times(0,\infty)$,  that is $\langle I_{\Omega_1^+}^{1}f,\widetilde{\mathbf{\Phi}}_{\tau_1}\widetilde{\mathbf{\Phi}}_{\tau_2}\rangle=0$ for all $\tau_2$ provided $\tau_1<0$ (see \eqref{vazhno}):
\begin{multline}\label{L3}
\sum_{\tau_1\in\mathbb{Z}}
\bigl|\langle I_{\Omega_1^+}^{1}f,\widetilde{\mathbf{\Phi}}_{\tau_1}\widetilde{\mathbf{\Phi}}_{\tau_2}\rangle\bigr|^pw_{0\tau_1}=\sum_{\tau_1\in\mathbb{Z}}
\bigl|\langle I_{\Omega_1^+}^{1}f,\bigl[\widetilde{\mathbf{\Phi}}_{\tau_1}\mp\widetilde{\mathbf{\Phi}}_{\tau_1-1}\bigr]\widetilde{\mathbf{\Phi}}_{\tau_2}\rangle\bigr|^pw_{0\tau_1}\\=:\sum_{\tau_1\in\mathbb{Z}}
\bigl|\langle I_{\Omega_1^+}^{1}f,\bigl[-\mu_{00\tau_1}+\widetilde{\mathbf{\Phi}}_{\tau_1-1}\bigr]\widetilde{\mathbf{\Phi}}_{\tau_2}\rangle\bigr|^pw_{0\tau_1}=\sum_{\tau_1\in\mathbb{Z}}
\bigl|\langle I_{\Omega_1^+}^{1}f,\bigl[-\mu_{00\tau_1}+\widetilde{\mathbf{\Phi}}_{\tau_1-1}\pm\widetilde{\mathbf{\Phi}}_{\tau_1-2}\bigr]\widetilde{\mathbf{\Phi}}_{\tau_2}\rangle\bigr|^pw_{0\tau_1}\\=:\sum_{\tau_1\in\mathbb{Z}}
\bigl|\langle I_{\Omega_1^+}^{1}f,\bigl[-\mu_{00\tau_1}-\mu_{00(\tau_1-1)}+\widetilde{\mathbf{\Phi}}_{\tau_1-2}\bigr]\widetilde{\mathbf{\Phi}}_{\tau_2}\rangle\bigr|^pw_{0\tau_1}=\ldots\\\ldots=\sum_{\tau_1\in\mathbb{Z}}
\bigl|\langle I_{\Omega_1^+}^{1}f,\bigl[-\mu_{00\tau_1}-\mu_{00(\tau_1-1)}-\ldots -\mu_{001}+\widetilde{\mathbf{\Phi}}_{0}\bigr]\widetilde{\mathbf{\Phi}}_{\tau_2}\rangle\bigr|^pw_{0\tau_1}\\=\sum_{\tau_1\ge 0}
\bigl|\langle I_{\Omega_1^+}^{1}f,\bigl[-\mu_{00\tau_1}-\mu_{00(\tau_1-1)}-\ldots -\mu_{001}-\mu_{000}\bigr]\widetilde{\mathbf{\Phi}}_{\tau_2}\rangle\bigr|^pw_{0\tau_1}.
\end{multline} Let $\boldsymbol{\mu}_{00\tau_1}:=\langle I_{\Omega_1^+}^{1}f,\mu_{00\tau_1}\widetilde{\mathbf{\Phi}}_{\tau_2}\rangle$, then, on the strength of \cite[Theorem 1.1]{PSU} and in view of \eqref{Md} for $d=0$, \begin{multline*}
\sum_{\tau_1\ge 0}
\bigl|\langle I_{\Omega_1^+}^{1}f,\bigl[-\mu_{00\tau_1}-\mu_{00(\tau_1-1)}-\ldots -\mu_{001}-\mu_{000}\bigr]\widetilde{\mathbf{\Phi}}_{\tau_2}\rangle\bigr|^pw_{0\tau_1}\\\le\sum_{\tau_1\ge 0}\Bigl(\sum_{0\le r\le \tau_1}\bigl|\boldsymbol{\mu}_{00r}
\bigr|\Bigr)^pw_{0\tau_1}\lesssim \bigl[\mathscr{M}_{\Omega_1^+}^1(0)\bigr]^p\ \sum_{\tau_1\ge 0}\bigl|\boldsymbol{\mu}_{00\tau_1}\bigr|^p\bar{u}_{0\tau_1}.
\end{multline*} Therefore,
\begin{equation*}\mathbf{V}_{I_{\Omega_1^+}^{1}f}^{(0)}\lesssim \bigl[\mathscr{M}_{\Omega_1^+}^1(0)\bigr]^p
\sum_{\tau_2\in\mathbb{Z}}
\sum_{\tau_1\in\mathbb{Z}}\bigl|\langle I_{\Omega_1^+}^{1}f,\bigl[\widetilde{\mathbf{\Phi}}_{\tau_1-1}-\widetilde{\mathbf{\Phi}}_{\tau_1}\bigr]\widetilde{\mathbf{\Phi}}_{\tau_2}\rangle\bigr|^p\int_{Q_{0\tau_1}}u_1\,\int_{Q_{0\tau_2}}w_2.
\end{equation*} In view of \eqref{L1} and \eqref{diff}
\begin{align}\label{L7}&\mathbf{\Lambda}'_2\mathbf{\Lambda}'_3\cdot\langle I_{\Omega_1^+}^{1}f,\bigl[\widetilde{\mathbf{\Phi}}_{\tau_1-1}-\widetilde{\mathbf{\Phi}}_{\tau_1}\bigr]\widetilde{\mathbf{\Phi}}_{\tau_2}\rangle\nonumber\\=&\int_0^\infty\int_0^\infty\biggl(\int_0^{x_1} f(y_1,y_2)\,dy_1\biggr)\Bigl[{\mathbf{\Phi}}_{2,-3}(x_1-\tau_1+1)-{\mathbf{\Phi}}_{2,-3}(x_1-\tau_1)\Bigr]{\mathbf{\Phi}}_{3,-4}(y_2-\tau_2)\,dx_1dy_2\nonumber\\=&\beta_2\int_0^\infty\int_0^\infty\biggl(\int_0^{x_1} f(y_1,y_2)\,dy_1\biggr)\Bigl[B_{2,-3}(x_1-\tau_1+1)-B_{2,-3}(x_1-\tau_1)\Bigr]{\mathbf{\Phi}}_{3,-4}(y_2-\tau_2)\,dx_1dy_2\nonumber\\\overset{\eqref{diff}}{=}&\beta_2\int_0^\infty\int_0^\infty f(y_1,y_2)\,B_{3,-3}(y_1-\tau_1+1){\mathbf{\Phi}}_{3,-4}(y_2-\tau_2)\,dy_1dy_2\nonumber\\=&:\frac{\beta_2}{\beta_3}\int_0^\infty\int_0^\infty f(y_1,y_2)\,{\mathbf{\Phi}}_{n_1^\ast=3,\boldsymbol{k}^\ast_1=-4}(y_1-\tau_1){\mathbf{\Phi}}_{3,-4}(y_2-\tau_2)\,dy_1dy_2\nonumber\\=&:\bigl(\mathbf{\Lambda}'_3\bigr)^2\frac{\beta_2}{\beta_3}\langle f,\widetilde{\mathbf{\Phi}}^\ast_{\tau_1}\widetilde{\mathbf{\Phi}}_{\tau_2}\rangle=:\bigl(\mathbf{\Lambda}'_3\bigr)^2\frac{\beta_2}{\beta_3}\langle f,{\mathbf{\Phi}}^\ast_{\tau}\rangle,\end{align} where ${\mathbf{\Phi}}^\ast_{\tau}$ is the product of $\widetilde{\mathbf{\Phi}}^\ast_{\tau_1}$ and $\widetilde{\mathbf{\Phi}}_{\tau_2}$ with 
$\widetilde{\mathbf{\Phi}}^\ast_{\tau_1}=\bigl(\mathbf{\Lambda}'_3\bigr)^{-1}{\mathbf{\Phi}}_{3,-4}(\cdot-\tau_1)$. We shall use this function 
to form another system $\{\mathbf{\Phi}^\ast_{\tau}\}$ than $\{\mathbf{\Phi}_{\tau}\}$ of the type \eqref{ForRepr'} suitable for representing $\|f\|_{{B}_{pq}^{1,u}(\mathbb{R}^2)}$, relatively to $\|I_{\Omega_1^+}^{1}f\|_{{B}_{pq}^{0,w}(\mathbb{R}^2)}$ in the sense of \eqref{ineqEx}, in an appropriate space ${b}_{pq}^{1,u}$ of sequences $\lambda^\ast$ with $u=u_1w_2$: \begin{equation}\label{seq0}\lambda^\ast_{00\tau}=\langle f,\mathbf{\Phi}^\ast_{\tau}\rangle\quad(\tau=(\tau_1,\tau_2)\in\mathbb{Z}^2),\qquad   \lambda^\ast_{id\tau}=2^{d}\langle f,\mathbf{\Psi}^\ast_{i(d-1)\tau}\rangle\quad(i=1,2,3;\,d\in\mathbb{N};\,\tau\in\mathbb{Z}^2).\end{equation} To establish $\{\mathbf{\Psi}^\ast_{i(d-1)\tau}\}$ we apply similar to \eqref{L3} procedure to each term in \eqref{L2}, again taking into account that $I_{\Omega_1^+}^{1}f=0$ on $\mathbb{R}^2\setminus\mathbb{R}^2_+$. On the strength of \eqref{vazhno} this means that $\langle I_{\Omega_1^+}^{1}f,\widetilde{\mathbf{\Psi}}_{1(d-1)\tau_1}\widetilde{\mathbf{\Psi}}_{1(d-1)\tau_2}\rangle=0$ for all $\tau_2$ provided $\tau_1<0$. Namely, we write for a fixed $\tau_2$, starting from the first term in \eqref{L2}: \begin{multline*}\sum_{\tau_1\in\mathbb{Z}}\bigl|\langle I_{\Omega_1^+}^{1}f,\widetilde{\mathbf{\Psi}}_{1(d-1)\tau_1}\widetilde{\mathbf{\Psi}}_{1(d-1)\tau_2}\rangle\bigr|^p w_{d\tau_1}\\=\sum_{\tau_1\in\mathbb{Z}}\bigl|\langle I_{\Omega_1^+}^{1}f,\bigl[\widetilde{\mathbf{\Psi}}_{1(d-1)\tau_1}\mp2\widetilde{\mathbf{\Psi}}_{1(d-1)(\tau_1-1)}\pm\widetilde{\mathbf{\Psi}}_{1(d-1)(\tau_1-2)}\bigr]\widetilde{\mathbf{\Psi}}_{1(d-1)\tau_2}\rangle\bigr|^p w_{d\tau_1}\\=:\sum_{\tau_1\in\mathbb{Z}}
\bigl|\langle I_{\Omega_1^+}^{1}f,\bigl[\nu_{1(d-1)\tau_1}+2\widetilde{\mathbf{\Psi}}_{1(d-1)(\tau_1-1)}-\widetilde{\mathbf{\Psi}}_{1(d-1)(\tau_1-2)}\bigr]\widetilde{\mathbf{\Psi}}_{1(d-1)\tau_2}\rangle\bigr|^pw_{d\tau_1}\\=\sum_{\tau_1\in\mathbb{Z}}
\bigl|\langle I_{\Omega_1^+}^{1}f,\bigl[\nu_{1(d-1)\tau_1}+2\widetilde{\mathbf{\Psi}}_{1(d-1)(\tau_1-1)}-\widetilde{\mathbf{\Psi}}_{1(d-1)(\tau_1-2)}\mp3\widetilde{\mathbf{\Psi}}_{1(d-1)(\tau_1-2)}\pm2\widetilde{\mathbf{\Psi}}_{1(d-1)(\tau_1-3)} \bigr]\widetilde{\mathbf{\Psi}}_{1(d-1)\tau_2}\rangle\bigr|^pw_{d\tau_1}\\=:\sum_{\tau_1\in\mathbb{Z}}
\bigl|\langle I_{\Omega_1^+}^{1}f,\bigl[\nu_{1(d-1)\tau_1}+2\nu_{1(d-1)(\tau_1-1)}+3\widetilde{\mathbf{\Psi}}_{1(d-1)(\tau_1-2)}-2\widetilde{\mathbf{\Psi}}_{1(d-1)(\tau_1-3)} \bigr]\widetilde{\mathbf{\Psi}}_{1(d-1)\tau_2}\rangle\bigr|^pw_{d\tau_1}=...\end{multline*}\begin{multline}\label{L8}\ldots=
\sum_{\tau_1\ge 0}
\bigl|\langle I_{\Omega_1^+}^{1}f,\bigl[\nu_{1(d-1)\tau_1}+2\nu_{1(d-1)(\tau_1-2)}+3\nu_{1(d-1)(\tau_1-3)}\ldots +(\tau_1+1)\nu_{1(d-1)0}\bigr]\widetilde{\mathbf{\Psi}}_{1(d-1)\tau_2}\rangle\bigr|^pw_{d\tau_1}\\\le
\sum_{\tau_1\ge 0}\biggl(\sum_{0\le r\le\tau_1}\Bigl|\langle I_{\Omega_1^+}^{1}f,(\tau_1-r+1)\nu_{1(d-1)r}\widetilde{\mathbf{\Psi}}_{1(d-1)\tau_2}\rangle\Bigr|\biggr)^pw_{d\tau_1}\\=:\sum_{\tau_1\ge 0}\biggl(\sum_{0\le r\le\tau_1}(\tau_1-r+1)\bigl|\boldsymbol{\nu}_{1(d-1)r}\bigr|\biggr)^pw_{d\tau_1}.\end{multline}
On the strength of \cite[Theorem 1.8
]{PSU} and in view of \eqref{Nd}, \begin{align}\label{L9}
\sum_{\tau_1\ge 0}\biggl(\sum_{0\le r\le\tau_1}(\tau_1-r+1)\bigl|\boldsymbol{\nu}_{1(d-1)r}\bigr|\biggr)^pw_{d\tau_1}&\lesssim\bigl[2^{2d}\,\mathscr{N}_{\Omega_1^+}^1(d)\bigr]^p\sum_{\tau_1\ge 0}\bigl|\boldsymbol{\nu}_{1(d-1)\tau_1}\bigr|^p\tilde{u}_{d\tau_1}\nonumber\\&=\bigl[2^{2d}\,\mathscr{N}_{\Omega_1^+}^1(d)\bigr]^p\sum_{\tau_1\ge 0}\bigl|\langle I_{\Omega_1^+}^{1}f,\nu_{1(d-1)\tau_1}\widetilde{\mathbf{\Psi}}_{1(d-1)\tau_2}\rangle\bigr|^p\tilde{u}_{d\tau_1}.\end{align} We have
\begin{align}\label{L5} 
&\mathbf{\Lambda}^{''}_2\mathbf{\Lambda}^{''}_3\langle I_{\Omega_1^+}^{1}f,\nu_{1(d-1)\tau_1}\widetilde{\mathbf{\Psi}}_{1(d-1)\tau_2}\rangle\nonumber\\=&\mathbf{\Lambda}^{''}_2\mathbf{\Lambda}^{''}_3\langle I_{\Omega_1^+}^{1}f,\bigl[\widetilde{\mathbf{\Psi}}_{1(d-1)(\tau_1-2)}-2\widetilde{\mathbf{\Psi}}_{1(d-1)(\tau_1-1)}+\widetilde{\mathbf{\Psi}}_{1(d-1)\tau_1}\bigr]\widetilde{\mathbf{\Psi}}_{1(d-1)\tau_2}\rangle\nonumber\\=&\int_0^\infty\int_0^\infty I_{\Omega_1^+}^{1}f(x_1,y_2)\Bigl[{\mathbf{\Psi}}_{2,3,-3,-7}(2^{d-1}x_1-\tau_1+2)-2{\mathbf{\Psi}}_{2,3,-3,-7}(2^{d-1}x_1-\tau_1+1)\nonumber\\&+{\mathbf{\Psi}}_{2,3,-3,-7}(2^{d-1}x_1-\tau_1)\Bigr]{\mathbf{\Psi}}_{3,\boldsymbol{m}_2=\emptyset,-4,-9}(2^{d-1}y_2-\tau_2)\,dx_1dy_2,\end{align} where for $\tau_1=0$ (see \eqref{L4}) \begin{multline}\label{L6}\hat{\mathbf{\Psi}}_{2,3,-3,-7}(\omega/2^{d-1})=\frac{\gamma_{2}} {2}\,\bigl|\mathcal{A}_{2}(\omega/2^d)\bigr|^2\,\bigl|\mathscr{A}_{3}(\omega/2^{d-1})\bigr|^{2}\sum_{k=0}^{3}\frac{(-1)^k(3)!}{k!(3-k)!}
\mathrm{e}^{(2-k)i\omega/2^d}\,\widehat{B}_{2}(\omega/2^d)\mathrm{e}^{-7i\omega/2^{d-1}}\\=\frac{\gamma_{2}} {2}\,\bigl|\mathcal{A}_{2}(\omega/2^d)\bigr|^2\,\bigl|\mathbf{A}_{3}(-\omega/2^d)\bigr|^{2}\bigl|\mathcal{A}_{3}(\omega/2^d)\bigr|^{2}\sum_{k=0}^{3}\frac{(-1)^k(3)!}{k!(3-k)!}
\mathrm{e}^{(2-k)i\omega/2^d}\,\widehat{B}_{2}(\omega/2^d)\mathrm{e}^{-7i\omega/2^{d-1}}.\end{multline} Notice that adding one more iteration to $\sum_{k=0}^{3}\frac{(-1)^k(3)!}{k!(3-k)!}
\mathrm{e}^{(2-k)i\omega/2^d}$ in \eqref{L6} contributes to forming
\begin{equation*}\hat{\mathbf{\Psi}}_{n_1^\ast=3,\boldsymbol{m}_1^\ast=\emptyset;\boldsymbol{k}_1^\ast =-4,\boldsymbol{\varkappa}_1^\ast=-5}(\omega/2^{d-1})=\frac{\gamma_{3}} {2}\,\bigl|\mathcal{A}_{3}(\omega/2^d)\bigr|^2\sum_{k=0}^{4}\frac{(-1)^k(4)!}{k!(4-k)!}
\mathrm{e}^{(3-k)i\omega/2^d}\,\widehat{B}_{3}(\omega/2^d)\mathrm{e}^{-5i\omega/2^{d-1}} \end{equation*} correspondingly to ${\mathbf{\Phi}}_{n_1^\ast=3,\boldsymbol{k}_1^\ast=-4}$ in \eqref{L7}. We have already added even two more difference iterations to our construction (see \eqref{L5}). The second one of them will be used for reducing 
$I_{\Omega_1^+}^{1}f$ to $f$. Indeed, by \eqref{diff},
\begin{multline}\label{L12}\int_0^\infty\int_0^\infty I_{\Omega_1^+}^{1}f(x_1,y_2)\Bigl[B_2(2^{d}x_1-\tau_1+2)-2B_2(2^{d}x_1-\tau_1+1)\\+B_2(2^{d}x_1-\tau_1)\Bigr]{\mathbf{\Psi}}_{3,\emptyset,-4,-9}(2^{d-1}y_2-\tau_2)\,dx_1dy_2\\=2^{-d}
\int_0^\infty\int_0^\infty f(y_1,y_2)\Bigl[B_3(2^{d}x_1-\tau_1+2)-B_3(2^{d}x_1-\tau_1+1)\Bigr]{\mathbf{\Psi}}_{3,\emptyset,-4,-9}(2^{d-1}y_2-\tau_2)\,dx_1dy_2.\end{multline} As for the product $\bigl|\mathcal{A}_{2}(\omega/2)\bigr|^2\,\bigl|\mathbf{A}_{3}(-\omega/2)\bigr|^{2}$ 
in \eqref{L6}, which can be re--written in the form
\begin{multline*}\bigl|\mathcal{A}_{2}(\omega/2)\bigr|^2\,\bigl|\mathbf{A}_{3}(-\omega/2)\bigr|^{2}=\bigl|1-\mathrm{e}^{i\omega/2}r_1(2)\bigr|^2\bigl|1-\mathrm{e}^{i\omega/2}r_2(2)\bigr|^2\\\times \bigl|1+\mathrm{e}^{i\omega/2}r_1(3)\bigr|^2\bigl|1+\mathrm{e}^{i\omega/2}r_2(3)\bigr|^2\bigl|1+\mathrm{e}^{i\omega/2}r_3(3)=:\sum_{j=-5}^5\boldsymbol{\lambda}_j\cdot \mathrm{e}^{ji\omega/2},\end{multline*} it will simply add to our construction of $\mathbf{\Psi}^\ast_{1(d-1)\tau}:=\widetilde{\mathbf{\Psi}}_{1(d-1)\tau_1}^\ast\widetilde{\mathbf{\Psi}}_{1(d-1)\tau_2}$ with $$\widetilde{\mathbf{\Psi}}_{1(d-1)\tau_1}^\ast=(\mathbf{\Lambda}^{''}_3)^{-1}{\mathbf{\Psi}}_{n_1^\ast=3,\boldsymbol{m}_1^\ast=\emptyset;\boldsymbol{k}_1^\ast =-4,\boldsymbol{\varkappa}_1^\ast=-5}(2^{d-1}\cdot-\tau_1)$$ a finite linear combination of half--integer shifts of ${\mathbf{\Psi}}_{n_1^\ast=3,\boldsymbol{m}_1^\ast=\emptyset;\boldsymbol{k}_1^\ast =-4,\boldsymbol{\varkappa}_1^\ast=-5}$ with some coefficients $\{\boldsymbol{\lambda}_j\}\subset\mathbb{R}$ depending of $r_j(2)$, $j=1,2$, and $r_j(3)$, $j=1,2,3$.
We obtain, by continuing the \eqref{L5} with respect to \eqref{L12}, 
\begin{align*}&\mathbf{\Lambda}^{''}_3
\mathbf{\Lambda}^{''}_2\langle I_{\Omega_1^+}^{1}f,\nu_{1(d-1)\tau_1}\widetilde{\mathbf{\Psi}}_{1(d-1)\tau_2}\rangle\\=&\frac{1}{2^{d}}\frac{\gamma_2}{\gamma_3}\sum_{j=-5}^5\boldsymbol{\lambda}_j\int_0^\infty\int_0^\infty f(y_1,y_2){\mathbf{\Psi}}_{n_1^\ast=3,\boldsymbol{m}_1^\ast=\emptyset;\boldsymbol{k}_1^\ast =-4,\boldsymbol{\varkappa}_1^\ast=-5}(y_1-\tau_1-j){\mathbf{\Psi}}_{3,\emptyset,-4,-9}(2^{d-1}x_2-\tau_2)\,dx_1dy_2\\=:& 
\frac{\bigl(\mathbf{\Lambda}^{''}_3\bigr)^2}{2^{d}}\frac{\gamma_2}{\gamma_3}\sum_{j=-5}^5\boldsymbol{\lambda}_j\cdot\langle f,\widetilde{\mathbf{\Psi}}^\ast_{1(d-1)(\tau_1+j)}\widetilde{\mathbf{\Psi}}_{1(d-1)\tau_2}\rangle=:\frac{\bigl(\mathbf{\Lambda}^{''}_3\bigr)^2}{2^{d}}\frac{\gamma_2}{\gamma_3}\sum_{j=-5}^5\boldsymbol{\lambda}_j\cdot\langle f,{\mathbf{\Psi}}^\ast_{1(d-1)(\tau_1+j,\tau_2)}\rangle.\end{align*} This means that the following estimate is valid  
\begin{equation*}
\bigl|\langle I_{\Omega_1^+}^{1}f,\nu_{1(d-1)\tau_1}\widetilde{\mathbf{\Psi}}_{1(d-1)\tau_2}\rangle\bigr|\le
2^{-d}\frac{\gamma_2\cdot\mathbf{\Lambda}^{''}_3}{\gamma_3\cdot\mathbf{\Lambda}^{''}_2}\sum_{j=-5}^5|\boldsymbol{\lambda}_j|\cdot\bigl|\langle f,{\mathbf{\Psi}}^\ast_{1(d-1)(\tau_1+j,\tau_2)}\rangle\bigr|.\end{equation*} This yields for the first term in \eqref{L2}, in combination with \eqref{L8} and \eqref{L9}, \begin{multline*}
\Bigl(\mathbf{W}_{I_{\Omega_1^+}^{1}f}^{1(d-1)}\Bigr)^\frac{1}{q}\lesssim 2^{2d}\,\mathscr{N}_{\Omega_1^+}^1(d)\biggl(\sum_{\tau_2\in\mathbb{Z}}\sum_{\tau_1\in\mathbb{Z}}\bigl|\langle I_{\Omega_1^+}^{1}f,\nu_{1(d-1)\tau_1}\widetilde{\mathbf{\Psi}}_{1(d-1)\tau_2}\rangle\bigr|^p
\int_{Q_{d\tau_1}}\tilde{u}_1\,\int_{Q_{d\tau_2}}w_2\biggr)^\frac{1}{p}\\\le 2^d\,{\mathscr{N}_{\Omega_1^+}^1(d)}\cdot\frac{\gamma_2\cdot\mathbf{\Lambda}^{''}_3}{\gamma_3\cdot\mathbf{\Lambda}^{''}_2}\,\sum_{j=-5}^5|\boldsymbol{\lambda}_j|\cdot\biggl(\sum_{\tau_1\in\mathbb{Z}}\sum_{\tau_2\in\mathbb{Z}}\bigl|\langle f,{\mathbf{\Psi}}^\ast_{1(d-1)(\tau_1+j,\tau_2)}\rangle\bigr|^p
\int_{Q_{d\tau_1}}\tilde{u}_1\,\int_{Q_{d\tau_2}}w_2\biggr)^\frac{1}{p}. \end{multline*} On the strength of \eqref{delta2},
$$\int_{Q_{d\tau_1}}\tilde{u}_1\le \int_{\cup_jQ_{d(\tau_1+j)}}\tilde{u}_1\lesssim \int_{Q_{d(\tau_1+j)}}\tilde{u}_1
.$$ Thus, we have for each $j$
\begin{equation*}
\sum_{\tau_1\in\mathbb{Z}}\sum_{\tau_2\in\mathbb{Z}}\bigl|\langle f,{\mathbf{\Psi}}^\ast_{1(d-1)(\tau_1+j,\tau_2)}\rangle\bigr|^p
\int_{Q_{d\tau_1}}\tilde{u}_1\,\int_{Q_{d\tau_2}}w_2\lesssim \sum_{\tau_1\in\mathbb{Z}}\sum_{\tau_2\in\mathbb{Z}}\bigl|\langle f,{\mathbf{\Psi}}^\ast_{1(d-1)(\tau_1+j,\tau_2)}\rangle\bigr|^p
\int_{Q_{d(\tau_1+j)}}\tilde{u}_1\,\int_{Q_{d\tau_2}}w_2,
\end{equation*} which implies
\begin{multline}\label{L10}
\Bigl(\mathbf{W}_{I_{\Omega_1^+}^{1}f}^{1(d-1)}\Bigr)^\frac{1}{q}\lesssim  2^d\,{\mathscr{N}_{\Omega_1^+}^1(d)}\cdot\frac{\gamma_2\cdot\mathbf{\Lambda}^{''}_3}{\gamma_3\cdot\mathbf{\Lambda}^{''}_2}\Bigl(\sum_{j=-5}^5|\boldsymbol{\lambda}_j|\Bigr
)\biggl(\sum_{\tau\in\mathbb{Z}^2}\bigl|\langle f,{\mathbf{\Psi}}^\ast_{1(d-1)\tau)}\rangle\bigr|^p
\int_{Q_{d\tau_1}}\tilde{u}_1\,\int_{Q_{d\tau_2}}w_2\biggr)^\frac{1}{p}\\\lesssim 
2^d\,{\mathscr{N}_{\Omega_1^+}^1(d)}\biggl(\sum_{\tau\in\mathbb{Z}^2}\bigl|\langle f,{\mathbf{\Psi}}^\ast_{1(d-1)\tau)}\rangle\bigr|^p
\int_{Q_{d\tau}}{u}\biggr)^\frac{1}{p}
.\end{multline}
Analogously, one can make a similar estimate for the second term in \eqref{L2} 
\begin{equation*}
\Bigl(\mathbf{W}_{I_{\Omega_1^+}^{1}f}^{2(d-1)}\Bigr)^\frac{1}{q}\lesssim  2^d\,{\mathscr{N}_{\Omega_1^+}^1(d)}\biggl(\sum_{\tau\in\mathbb{Z}^2}\bigl|\langle f,{\mathbf{\Psi}}^\ast_{2(d-1)\tau)}\rangle\bigr|^p
\int_{Q_{d\tau}}{u}\biggr)^\frac{1}{p}\end{equation*}
with $\mathbf{\Psi}^\ast_{2(d-1)\tau}:=\widetilde{\mathbf{\Psi}}_{2(d-1)\tau_1}^\ast\widetilde{\mathbf{\Phi}}_{2(d-1)\tau_2}$, where $\widetilde{\mathbf{\Psi}}_{2(d-1)\tau_1}^\ast=\widetilde{\mathbf{\Psi}}_{1(d-1)\tau_1}^\ast$. The estimate for the third term
\begin{equation*}
\Bigl(\mathbf{W}_{I_{\Omega_1^+}^{1}f}^{3(d-1)}\Bigr)^\frac{1}{q}\lesssim  {\mathscr{M}_{\Omega_1^+}^1(d)}\biggl(\sum_{\tau\in\mathbb{Z}^2}\bigl|\langle f,{\mathbf{\Psi}}^\ast_{3(d-1)\tau)}\rangle\bigr|^p
\int_{Q_{d\tau}}{u}\biggr)^\frac{1}{p}\le 2^d\,{\mathscr{M}_{\Omega_1^+}^1(d)}\biggl(\sum_{\tau\in\mathbb{Z}^2}\bigl|\langle f,{\mathbf{\Psi}}^\ast_{3(d-1)\tau)}\rangle\bigr|^p
\int_{Q_{d\tau}}{u}\biggr)^\frac{1}{p}\end{equation*}
with $\mathbf{\Psi}^\ast_{3(d-1)\tau}:=\widetilde{\mathbf{\Phi}}_{3(d-1)\tau_1}^\ast\widetilde{\mathbf{\Psi}}_{3(d-1)\tau_2}$, where $$\widetilde{\mathbf{\Phi}}_{3(d-1)\tau_1}^\ast=(\mathbf{\Lambda}^{'}_3)^{-1}{\mathbf{\Phi}}_{n_1^\ast=3,\boldsymbol{k}_1^\ast =-4}(2^{d-1}\cdot-\tau_1),$$ 
can be made similarly to \eqref{L7}, by adding scaling parameter $d\in\mathbb{N}$ into the ${\mathbf{\Phi}}_{n_1^\ast=3,\boldsymbol{k}_1^\ast =-4}$. We now have
\begin{equation*}\sum_{i=1}^3\mathbf{W}_{I_{\Omega_1^+}^{1}f}^{i(d-1)}\lesssim 2^{d(1-2/p)q}\,{\Bigl(\max\bigl\{\mathscr{M}_{\Omega_1^+}^1(d),\mathscr{N}_{\Omega_1^+}^1(d)\bigr\}\Bigr)^q}
\sum_{i=1}^{3}\biggl(\int_{\mathbb{R}^2}\Bigl(\sum_{\tau\in\mathbb{Z}^2} \bigl|\langle f,\mathbf{\Psi}^\ast_{i(d-1)\tau}\rangle\bigr|\chi_{d\tau}^{(p)}(x)\Bigr)^p u(x)\,dx\biggr)^\frac{q}{p}\end{equation*} for $d\in\mathbb{N}$, that is \begin{equation*}\Bigl(\mathbf{W}_{I_{\Omega_1^+}^{1}f}\Bigr)^q\lesssim \bigl[C_{\Omega_1^+}^{1}\bigr]^q\sum_{d\in\mathbb{N}}2^{d(2-2/p)q}
\sum_{i=1}^{3}\biggl(\int_{\mathbb{R}^2}\Bigl(\sum_{\tau\in\mathbb{Z}^2} \bigl|\langle f,\mathbf{\Psi}^\ast_{i(d-1)\tau}\rangle\bigr|\chi_{d\tau}^{(p)}(x)\Bigr)^p u(x)\,dx\biggr)^\frac{q}{p}\end{equation*} and \begin{equation*}
\Bigl(\mathbf{V}_{I_{\Omega_1^+}^{1}f}\Bigr)^p\lesssim\bigl[\mathscr{M}_{\Omega_1^+}^1(0)\bigr]^p\int_{\mathbb{R}^2}\Bigl(\sum_{\tau\in\mathbb{Z}^2}\bigl|\langle f,\mathbf{\Phi}^\ast_{\tau}\rangle\bigr|\chi_{0\tau}^{(p)}(x)\Bigr)^pu(x)\,dx.\end{equation*} This yields, respectively to \eqref{L11}, with $\lambda^\ast$ of the form \eqref{seq0},
\begin{multline*}\|\lambda\|_{b^{0,w}_{pq}}\lesssim \mathscr{M}_{\Omega_1^+}^1(0)\biggl(\int_{\mathbb{R}^2}\Bigl(\sum_{\tau\in\mathbb{Z}^2}\bigl|\langle f,\mathbf{\Phi}^\ast_{\tau}\rangle\bigr|\chi_{0\tau}^{(p)}(x)\Bigr)^pu(x)\,dx\biggr)^{\frac{1}{p}}\\+
C_{\Omega_1^+}^{1}\Biggl(\sum_{d=1}^\infty 2^{d(2-2/p)q}\sum_{i=1}^{3}\biggl(\int_{\mathbb{R}^2}\Bigl(\sum_{\tau\in\mathbb{Z}^2} \bigl|\langle f,\mathbf{\Psi}^\ast_{i(d-1)\tau}\rangle\bigr|\chi_{d\tau}^{(p)}(x)\Bigr)^p u(x)\,dx\biggr)^\frac{q}{p}\Biggr)^\frac{1}{q}\lesssim C_{\Omega_1^+}^{1}\|\lambda^\ast\|_{b^{1,u}_{pq}},\end{multline*} and the validity of \eqref{ineqEx} is now performed basing on the statement of Theorem \ref{main'}.\end{example}

Similar to \eqref{ineqEx} estimates for some particular cases of Triebel--Lizorkin spaces were considered in \cite{Oin2011}.

Observe that the conditions \eqref{Md} and \eqref{Nd} can be significantly released for many cases of weights $u$ and $w$. In particular, if $p=2$ and $\int_{Q_{d\tau}}v_1\approx\,2^{-d}v_1(\tau/2^d)$ for the both $v_1=u_1$ and $v_1=w_1$, then $$\mathscr{M}_{\Omega_1^+}^1(d)\equiv \mathscr{M}_{\Omega_1^+}^1(0)\quad\textrm{and}\quad\mathscr{N}_{\Omega_1^+}^1({d+1})\equiv \mathscr{N}_{\Omega_1^+}^1(1)\qquad\forall d\in\mathbb{N}_0.$$ Moreover, $\mathscr{N}_{\Omega_1^+}^1({d})<\infty$ yields $\mathscr{M}_{\Omega_1^+}^1({d})<\infty$ if $w(x_1)=(1+|x_1|)^{-\alpha}$, $u(x_1)=(1+|x_1|)^{p+2-\alpha}$, $\alpha>p+2$.

In general case, the next two theorems regulate relations between norms of images and pre--images of integration operators $I_{\Omega_k^\pm}^{m_k}$ of order $m_k$ in $A^{s,w}_{pq}(\mathbb{R}^N)$ with respect to variable $x_k$, $k\in\{1,\ldots,N\}$. 

\begin{theorem}\label{ImagesA} Let $1<p<\infty$, $0<q\le\infty$, $s\in\mathbb{R}$, weights $u,w\in\mathscr{A}_\infty^\loc$ and $f\in L_1^\textrm{loc}(\mathbb{R}^N)$. For $m_l\in\mathbb{N}$ and $c_l\in\mathbb{R}$ with $l\in\{1,\ldots,N\}$ let $I_{\Omega_l^+}^{m_l}$ be defined by \eqref{left}. Suppose that $f(x_{x_l})\equiv 0$ for $x_l\in(-\infty,c_l)$. \\ {\rm (i)} Assume $u,w$ are of product type (factorisable), that is $v(x_1,\ldots,x_N)=v_1(x_1)\cdot\ldots\cdot v_N(x_N)$ for the both $v=u$ and $v=w$ with one variable functions $v_l(x_l)$, $l=1,\ldots,N$; besides, $u(x_1,\ldots,x_N)=w_1(x_1)\cdot\ldots u_l(x_l)\ldots\cdot w_N(x_N)$. Then
$I_{\Omega_l^+}^{m_l}f\in {B}_{pq}^{s,w}(\mathbb{R}^N)$ if $f\in {B}_{pq}^{s+m_l,u}(\mathbb{R}^N)$ provided
\begin{align}\label{Md+}\mathscr{M}_{\Omega_l^+}^{m_l}(d):=&\frac{1}{2^{dm_l}}\Biggl[\sup_{\tau_l\ge [c_l]}\biggl(\sum_{r\ge\tau_l}(r-\tau_l+1)^{p(m_l-1)}\int_{Q_{dr}^{[c_l]}}w_l\biggr)^{\frac{1}{p}}\biggl(\sum_{[c_l]\le r\le\tau_l} \Bigl(\int_{Q_{dr}^{[c_l]}}\bar{u}_l\Bigr)^{1-p'}\biggr)^{\frac{1}{p'}}\\+&\sup_{\tau_l\ge [c_l]}\biggl(\sum_{r\ge\tau_l}\int_{Q_{dr}^{[c_l]}}w_l\biggr)^{\frac{1}{p}}\biggl(\sum_{[c_l]\le r\le\tau_l}(\tau_l-r+1)^{p'(m_l-1)} \Bigl(\int_{Q_{dr}^{[c_l]}}\bar{u}_l\Bigr)^{1-p'}\biggr)^{\frac{1}{p'}}\Biggr]<\infty\quad\forall\ d\in\mathbb{N}_0,\nonumber
\\\label{Nd+}\mathscr{N}_{\Omega_l^+}^{m_l}(d):=&\frac{1}{2^{2dm_l}}\Biggl[\sup_{\tau_l\ge [c_l]}\biggl(\sum_{r\ge\tau_l}(r-\tau_l+1)^{p(2m_l-1)}\int_{Q_{dr}^{[c_l]}}w_l\biggr)^{\frac{1}{p}}\biggl(\sum_{[c_l]\le r\le\tau_l} \Bigl(\int_{Q_{dr}^{[c_l]}}\tilde{u}_l\Bigr)^{1-p'}\biggr)^{\frac{1}{p'}}\\&+\sup_{\tau_l\ge [c_l]}\biggl(\sum_{r\ge\tau_l}\int_{Q_{dr}^{[c_l]}}w_l\biggr)^{\frac{1}{p}}\biggl(\sum_{[c_l]\le r\le\tau_l} (\tau_l-r+1)^{p'(2m_l-1)}\Bigl(\int_{Q_{dr}^{[c_l]}}\tilde{u}_l\Bigr)^{1-p'}\biggr)^{\frac{1}{p'}}\Biggr]<\infty\quad\forall \ d\in\mathbb{N},\nonumber
\end{align} where ${Q_{dr}^{[c_l]}}:=\Bigl[{\frac{r-[c_l]-1/2}{2^d}},{\frac{r-[c_l]+1/2}{2^d}}\Bigr]$,
$\bar{u_l}\le u_l$ and $\tilde{u}_l\le u_l$ on $\mathbb{R}$. Moreover, \begin{equation}\label{more+}\|I_{\Omega_l^+}^{m_l}f\|_{{B}_{pq}^{s,w}(\mathbb{R}^N)}\lesssim C_{\Omega_l^+}^{m_l}\|f\|_{{B}_{pq}^{s+m_l,u}(\mathbb{R}^N)},\qquad\textrm{where }\ C_{\Omega_l^+}^{m_l}:=\sup_{d\in\mathbb{N}_0}\bigl[\mathscr{M}_{\Omega_l^+}^{m_l}(d)+\mathscr{N}_{\Omega_l^+}^{m_l}(d+1)\bigr].\end{equation}
{\rm (ii)} If $I_{\Omega_l^+}^{m_l}f\in {A}_{pq}^{s,w}(\mathbb{R}^N)$ then $f\in {A}_{pq}^{s-m_l,w}(\mathbb{R}^N)$, besides,
$$\|f\|_{{A}_{pq}^{s-m_l,w}(\mathbb{R}^N)}\lesssim\|I_{\Omega_l^+}^{m_l}f\|_{{A}_{pq}^{s,w}(\mathbb{R}^N)}.$$
\end{theorem}
\begin{proof}
To perform the part (i) we assume that $f\in{B}_{pq}^{s+m_l,u}(\mathbb{R}^N)$ and choose a natural number $n^\ast_0$ satisfying \eqref{condB} with respect to the both weights $w$ and $u$. 
Then, on the strength of Theorem \ref{main'},
\begin{equation}\label{krysha}\|I_{\Omega_l^+}^{m_l}f\|_{B^{s,w}_{pq}(\mathbb{R}^N)}\approx\|\lambda\|_{b_{pq}^{s,w}}\end{equation} with $\lambda=\{\lambda_{00\tau}\}_{\tau\in\mathbb{Z}^N}\cup\{\lambda_{id\tau}\}_{{i=1,\ldots,2^N-1};\,{d\in\mathbb{N};\,\tau\in\mathbb{Z}^N}}$ of the form \eqref{asty}, with $I_{\Omega_l^+}^{m_l}f$ instead of $f$, where the system \eqref{ForRepr'} of functions $\mathbf{\Phi}_{\tau}$ and $\mathbf{\Psi}_{i(d-1)\tau}$ is generated by 
$N$ number of one--dimensional systems \eqref{vazhnoo}. For our convenience, we assume that \eqref{chi} in the definition of $b_{pq}^{s,w}$ is the $p-$normalised characteristic function of the cube $Q_{d(\tau_1,\ldots,\tau_l-[c_l],\ldots,\tau_N)}$ having the $l-$th shifted component in comparison with $Q_{d\tau}$.
 
We choose parameters in \eqref{vazhnoo} for all $k=1,\ldots,N$ as follows: $n_l=n_0^\ast$, $n_k=n_0^\ast+m_l$ for $k\not= l$, $$\boldsymbol{\kappa}_l:=\min\Bigl\{r\le[c_l]\colon [c_l,+\infty)\cap \textrm{supp}\,\mathbf{\Phi}_{n_l=n_0^\ast,r}\not=\emptyset\Bigr\}\qquad\textrm{and}\quad\boldsymbol{\kappa}_k=0\quad (k\not=l),$$
$$\boldsymbol{\varkappa}_l:=\min\Bigl\{r\le[c_l]\colon [c_l,+\infty)\cap \textrm{supp}\,\mathbf{\Psi}_{n_l=n_0^\ast,\boldsymbol{m}_l(\Bbbk_l),\boldsymbol{\kappa}_l,r}\not=\emptyset\Bigr\}\qquad\textrm{and}\quad\boldsymbol{\varkappa}_k=0\quad (k\not=l).$$ 
We suppose all $\Bbbk_k=0$ except $\Bbbk_l=1$ and put $\boldsymbol{m}_l(\Bbbk_l)=n_0^\ast+m_l$. This means that for each of the $N$ variables $x_k$, $k=1,\ldots,N$, we operate with one variable functions ${\mathbf{\Phi}}_{n_k,\boldsymbol{k}_k}$ and ${\mathbf{\Psi}}_{n_k,\boldsymbol{m}_k(\Bbbk_k),\boldsymbol{k}_k,\boldsymbol{\varkappa}_k}$ such that (see \eqref{MMM0} and \eqref{MMM}) \begin{equation*}\label{L1'}\widehat{\mathbf{\Phi}}_{n_k,\boldsymbol{k}_k}(\omega)=\widehat{\mathbf{\Phi}}_{r_1,\ldots,r_{n_k};\boldsymbol{k}_k}(\omega)=\beta_{{n_k}}\,\widehat{B}_{{n_k},\boldsymbol{\kappa}_k}(\omega)\end{equation*} and 
\begin{multline*}
\widehat{\mathbf{\Psi}}_{n_k,\boldsymbol{m}_k(\Bbbk_k),\boldsymbol{k}_k,\boldsymbol{\varkappa}_k}(\omega):=\widehat{\mathbf{\Psi}}_{r_1,\ldots,r_{n_k};\boldsymbol{m}_k(\Bbbk_k),\boldsymbol{k}_k,\boldsymbol{\varkappa}_k}(\omega)\nonumber\\=
\frac{\gamma_{n_k}} {2}\,\bigl|\mathcal{A}_{{n_k}}(\omega/2)\bigr|^2\,\bigl|\mathscr{A}_{\boldsymbol{m}_k}(\omega)\bigr|^{2\Bbbk_k}\sum_{l=0}^{{n_k}+1}\frac{(-1)^l({n_k}+1)!}{l!({n_k}+1-l)!}
\mathrm{e}^{({n_k}-l)i\omega/2}\,\widehat{B}_{{n_k},\boldsymbol{\varkappa}_k}(\omega/2).\end{multline*} 
Further considerations are similar to those in Example \ref{Exam}. Starting from the right hand side of \eqref{krysha} one should estimate it from above by $\|\lambda^\ast\|_{b_{pq}^{s+m_l,u}}$ with \begin{equation*}\label{seq0'}\lambda^\ast_{00\tau}=\langle f,\mathbf{\Phi}^\ast_{\tau}\rangle\quad(\tau\in\mathbb{Z}^N),\qquad   \lambda^\ast_{id\tau}=2^{d}\langle f,\mathbf{\Psi}^\ast_{i(d-1)\tau}\rangle\quad(i=1,\ldots,2^{N}-1;\,d\in\mathbb{N};\,\tau\in\mathbb{Z}^N),\end{equation*} where functions $\mathbf{\Phi}^\ast_{\tau}$ and $\mathbf{\Psi}_{i(d-1)\tau}^\ast$ differ from $\mathbf{\Phi}_{\tau}$ and $\mathbf{\Psi}_{i(d-1)\tau}$ by the components at the $l-$th place only:
$$\widetilde{\mathbf{\Phi}}^\ast_{\tau_l}=\bigl(\mathbf{\Lambda}'_{{n_0^\ast}+m_l}\bigr)^{-1}{\mathbf{\Phi}}_{n_0^\ast+m_l,\boldsymbol{k}_l^\ast=-\boldsymbol{\kappa}_l-m_l}(\cdot-\tau_l)$$ and $$
\widetilde{\mathbf{\Psi}}_{i(d-1)\tau_l}^\ast=\bigl(\mathbf{\Lambda}^{''}_{{n_0^\ast}+m_l}\bigr)^{-1}2^{(d-1)/N}{\mathbf{\Psi}}_{n_l^\ast+m_l,\boldsymbol{m}_l^\ast=\emptyset,\boldsymbol{k}_l^\ast =-\boldsymbol{\kappa}_l-m_l,\boldsymbol{\varkappa}_l^\ast=-\boldsymbol{\varkappa}_l-2m_l}(2^{d-1}\cdot-\tau_1)
$$ or $$
\widetilde{\mathbf{\Phi}}_{i(d-1)\tau_l}^\ast=\bigl(\mathbf{\Lambda}'_{{n_0^\ast}+m_l}\bigr)^{-1}2^{(d-1)/N}{\mathbf{\Phi}}_{n_l^\ast+m_l,\boldsymbol{k}_l^\ast =-\boldsymbol{\kappa}_l-m_l}(2^{d-1}\cdot-\tau_1).
$$ From this (i) follows by applying Theorem \ref{main'} with $n_0=n_0^\ast+m_l$.

The part (ii) for natural $s$ follows from \eqref{DifInt}.
To prove this assertion for $s\in\mathbb{R}$ assume that $I_{\Omega_l^+}^{m_l}f\in {A}_{pq}^{s,w}(\mathbb{R}^N)$ and fix some natural number $n_\ast$ fitting the condition \eqref{condB} with respect to $w$ in the case ${A}_{pq}^{s,w}(\mathbb{R}^N)={B}_{pq}^{s,w}(\mathbb{R}^N)$ or satisfying \eqref{condF} if ${A}_{pq}^{s,w}(\mathbb{R}^N)={F}_{pq}^{s,w}(\mathbb{R}^N)$. By Theorem \ref{main'},
\begin{equation*}\|I_{\Omega_l^+}^{m_l}f\|_{A^{s,w}_{pq}(\mathbb{R}^N)}\approx\|\lambda\|_{a_{pq}^{s,w}}\end{equation*} with $\lambda=\{\lambda_{00\tau}\}_{\tau\in\mathbb{Z}^N}\cup\{\lambda_{id\tau}\}_{{i=1,\ldots,2^N-1};\,{d\in\mathbb{N};\,\tau\in\mathbb{Z}^N}}$, defined by \eqref{asty} for $I_{\Omega_l^+}^{m_l}f$ instead of $f$, and represented in terms of functions $\mathbf{\Phi}_{\tau}$ and $\mathbf{\Psi}_{i(d-1)\tau}$ generated by 
$N$ number of one--dimensional systems \eqref{vazhnoo} with $\Bbbk_k=0$ for all $k=1,\ldots,N$ and $n_k=n_\ast$ for any $k=1,\ldots,N$. These one--variable functions ${\mathbf{\Phi}}_{n_k,\boldsymbol{\kappa}_k}$ and ${\mathbf{\Psi}}_{n_k;\boldsymbol{m}_k(\Bbbk_k),\boldsymbol{\kappa}_k,\boldsymbol{\varkappa}_k}$, originating $\mathbf{\Phi}_{\tau}$ and $\mathbf{\Psi}_{i(d-1)\tau}$, have the following forms: \begin{equation}\label{L1''}\widehat{\mathbf{\Phi}}_{n_k=n_\ast,\boldsymbol{\kappa}_k}(\omega)=\widehat{\mathbf{\Phi}}_{r_1,\ldots,r_{n_\ast};\boldsymbol{\kappa}_k}(\omega)=\beta_{{n_\ast}}\,\widehat{B}_{{n_\ast},\boldsymbol{\kappa}_k}(\omega),\end{equation} 
\begin{align}\label{L4''}
\widehat{\mathbf{\Psi}}_{n_k={n_\ast},\boldsymbol{m}_k(\Bbbk_k)=\emptyset,\boldsymbol{\kappa}_k,\boldsymbol{\varkappa}_k}(\omega):&=\widehat{\mathbf{\Psi}}_{r_1,\ldots,r_{n_\ast};\boldsymbol{m}_k(\Bbbk_k)=\emptyset,\boldsymbol{\kappa}_k,\boldsymbol{\varkappa}_k}(\omega)\nonumber\\&=
\frac{\gamma_{n_\ast}} {2}\,\bigl|\mathcal{A}_{{n_\ast}}(\omega/2)\bigr|^2\,\sum_{k=0}^{{n_\ast}+1}\frac{(-1)^k({n_\ast}+1)!}{k!({n_\ast}+1-k)!}
\mathrm{e}^{({n_\ast}-k)i\omega/2}\,\widehat{B}_{{n_\ast},\boldsymbol{\varkappa}_k}(\omega/2).\end{align}
Denote $\mathbf{\Psi}_{0(-1)\tau}:=\mathbf{\Phi}_\tau$. By construction, any of $\mathbf{\Psi}_{i(d-1)\tau}$, $i=0,\ldots, 2^N-1$, $d\in\mathbb{N}_0$, $\tau\in\mathbb{Z}^N$ in \eqref{ForRepr'} is an $N-$terms product of $\widetilde{\mathbf{\Phi}}_{n_\ast,\boldsymbol{\kappa}_k}(2^dx_k-\tau_{k})$ and/or $\widetilde{\mathbf{\Psi}}_{n_\ast;\emptyset,\boldsymbol{\kappa}_m,\boldsymbol{\varkappa}_m}(2^dx_m-\tau_{m})$, where $k,m$ take different values from $\{1,\ldots,N\}$, $\boldsymbol{\kappa}_k,\boldsymbol{\varkappa}_k,\boldsymbol{\kappa}_m,\boldsymbol{\varkappa}_m$ are fixed and $\tau_{k},\tau_{m}$ run over integers. We assume $\boldsymbol{\kappa}_k=\boldsymbol{\varkappa}_k=\boldsymbol{\kappa}_m=\boldsymbol{\varkappa}_m=0$ for all $k,m\not=l$, while $\boldsymbol{\kappa}_l$ and $\boldsymbol{\varkappa}_l$ are supposed to be arbitrary.

Our goal is to obtain an estimate from below of the norm $\|\lambda\|_{a_{pq}^{s,w}}$ by another norm $\|\lambda_\ast\|_{a_{pq}^{s+m_l,w}}$ with $\lambda_\ast=\{\lambda_{00\tau}\}_{\tau\in\mathbb{Z}^N}\cup\{\lambda_{id\tau}\}_{{i=1,\ldots,2^N-1};\,{d\in\mathbb{N};\,\tau\in\mathbb{Z}^N}}$, which is  different to $\lambda$ and is representing not $I_{\Omega_l^+}^{m_l}f$ but $f$ in $A^{s+m_l,w}_{p,q}(\mathbb{R}^N)$. The $\lambda_\ast$, in comparison to $\lambda$, are formulated in terms of $\mathbf{\Psi}_{i(d-1)\tau}$, $i=0,\ldots, 2^N-1$, $d\in\mathbb{N}_0$, $\tau\in\mathbb{Z}^N$, having functions $\widetilde{\mathbf{\Phi}}_{n_\ast+m_l,\boldsymbol{\kappa}_l-m_l}(2^dx_l-\tau_{l})$ and/or $\widetilde{\mathbf{\Psi}}_{n_\ast+m_l;\boldsymbol{m}_l(\Bbbk_l)=n_\ast,\boldsymbol{\kappa}_l-m_l,\boldsymbol{\varkappa}_l-2m_l}(2^dx_l-\tau_{l})$ instead of $\widetilde{\mathbf{\Phi}}_{n_\ast,\boldsymbol{\kappa}_l}(2^dx_l-\tau_{l})$ and/or $\widetilde{\mathbf{\Psi}}_{n_\ast;\emptyset,\boldsymbol{\kappa}_l,\boldsymbol{\varkappa}_l}(2^dx_l-\tau_{l})$ staying at the $l-$th place:
\begin{equation}\label{L1'''}\widehat{\mathbf{\Phi}}_{n_\ast+m_l,-m_l}(\omega)=\widehat{\mathbf{\Phi}}_{r_1,\ldots,r_{n_\ast+m_l};-m_l}(\omega)=\beta_{{n_\ast+m_l}}\,\widehat{B}_{{n_\ast+m_l},-m_l}(\omega),\end{equation} 
\begin{align}\label{L4'''}
\widehat{\mathbf{\Psi}}_{{n_\ast}+m_l,\boldsymbol{m}_l(\Bbbk_l)=n_\ast,-m_l,-2m_l}(\omega):=&\widehat{\mathbf{\Psi}}_{r_1,\ldots,r_{n_\ast+m_l};\boldsymbol{m}_l(\Bbbk_l)=n_\ast,-m_l,-2m_l}(\omega)\nonumber\\=&
\frac{\gamma_{n_\ast+m_l}} {2}\,\bigl|\mathcal{A}_{{n_\ast}+m_l}(\omega/2)\bigr|^2\,\bigl|\mathscr{A}_{{n_\ast}}(\omega)\bigr|^2\\&\times\sum_{k=0}^{{n_\ast}+m_l+1}\frac{(-1)^k({n_\ast}+m_l+1)!}{k!({n_\ast}+m_l+1-k)!}
\mathrm{e}^{({n_\ast}+m_l+-k)i\omega/2}\,\widehat{B}_{{n_\ast}+m_l,-2m_l}(\omega/2).\nonumber\end{align} The rest one--variable functions constituting $\mathbf{\Psi}_{i(d-1)\tau}$ in $\lambda_\ast$ are the same as in $\lambda$.
Procedure of obtaining the required estimate is similar to the proof of the part (i), but should be done in the reverse direction. Therefore, it does not require fulfilment of the  condition $C_{\Omega_l^+}^{m_l}<\infty$. The result can be achieved by applying Theorem \ref{main'} finite (dependently on $n_\ast$ and $m_l$) number times, in order to make necessary 
linear combinations of integer-- or half--shifts of ${\mathbf{\Phi}}_{n_k=n_\ast,\boldsymbol{\kappa}_k}$ and ${\mathbf{\Psi}}_{n_k={n_\ast},\boldsymbol{m}_k(\Bbbk_k)=\emptyset,\boldsymbol{\kappa}_k,\boldsymbol{\varkappa}_k}$ (varying parameters $\boldsymbol{\kappa}_l$ and $\boldsymbol{\varkappa}_l$ in \eqref{L1''} and \eqref{L4''}) resulting in \eqref{L1'''} and \eqref{L4'''} by integrating $B_{n_\ast}$ with respect to the $l-$th variable.
\end{proof}

Analogously, one can prove the following
\begin{theorem}\label{ImagesB} Let $1<p<\infty$, $0<q\le\infty$, $s\in\mathbb{R}$, weights $u,w\in\mathscr{A}_\infty^\loc$ and $f\in L_1^\textrm{loc}(\mathbb{R}^N)$. For $m_n\in\mathbb{N}$ and $c_n\in\mathbb{R}$, $n\in\{1,\ldots,N\}$, assume $I_{\Omega_n^-}^{m_n}$ is defined by \eqref{right}. Suppose that $f(x_{x_n})\equiv 0$ for $x_n\in(c_n,+\infty)$.
\\ {\rm (i)} Assume $u,w$ are of product type (factorisable), that is $v(x_1,\ldots,x_N)=v_1(x_1)\cdot\ldots\cdot v_N(x_N)$ for the both $v=u$ and $v=w$ with one variable functions $v_l(x_l)$, $l=1,\ldots,N$; besides, $u(x_1,\ldots,x_N)=w_1(x_1)\ldots u_n(x_n)\ldots w_N(x_N)$. Then
$I_{\Omega_n^-}^{m_n}f\in {B}_{pq}^{s,w}(\mathbb{R}^N)$ if $f\in {B}_{pq}^{s+m_n,u}(\mathbb{R}^N)$ provided
\begin{align}\label{Md-}\mathscr{M}_{\Omega_n^-}^{m_n}(d):=&\frac{1}{2^{dm_n}}\Biggl[\sup_{\tau_n\le [c_n]}\biggl(\sum_{r\le\tau_n}(\tau_n-r+1)^{p(m_n-1)}\int_{Q_{dr}^{[c_n]}}w_n\biggr)^{\frac{1}{p}}\biggl(\sum_{\tau_n\le r\le [c_n]} \Bigl(\int_{Q_{dr}^{[c_n]}}\bar{u}_n\Bigr)^{1-p'}\biggr)^{\frac{1}{p'}}\\+&\sup_{\tau_n\le [c_n]}\biggl(\sum_{r\le\tau_n}\int_{Q_{dr}^{[c_n]}}w_n\biggr)^{\frac{1}{p}}\biggl(\sum_{\tau_n\le r\le [c_n]}(r-\tau_n+1)^{p'(m_n-1)} \Bigl(\int_{Q_{dr}^{[c_n]}}\bar{u}_n\Bigr)^{1-p'}\biggr)^{\frac{1}{p'}}\Biggr]<\infty\quad\forall\ d\in\mathbb{N}_0,\nonumber\\\mathscr{N}_{\Omega_n^-}^{m_n}(d):=&\frac{1}{2^{2dm_n}}\Biggl[\sup_{\tau_n\le [c_n]}\biggl(\sum_{r\le\tau_n}(\tau_n-r+1)^{p(2m_n-1)}\int_{Q_{dr}^{[c_n]}}w_n\biggr)^{\frac{1}{p}}\biggl(\sum_{\tau_n\le r\le [c_n]} \Bigl(\int_{Q_{dr}^{[c_n]}}\tilde{u}_n\Bigr)^{1-p'}\biggr)^{\frac{1}{p'}}\label{Nd-}\\&+\sup_{\tau_n\le [c_n]}\biggl(\sum_{r\le\tau_n}\int_{Q_{dr}^{[c_n]}}w_n\biggr)^{\frac{1}{p}}\biggl(\sum_{\tau_n\le r\le [c_n]} (r-\tau_n+1)^{p'(2m_n-1)}\Bigl(\int_{Q_{dr}^{[c_n]}}\tilde{u}_n\Bigr)^{1-p'}\biggr)^{\frac{1}{p'}}\Biggr]<\infty\quad\forall\ d\in\mathbb{N},\nonumber\end{align} where ${Q_{dr}^{[c_n]}}:=\Bigl[{\frac{r-[c_n]-1/2}{2^d}},{\frac{r-[c_n]+1/2}{2^d}}\Bigr]$, 
$\bar{u_n}\le u_n$ and $\tilde{u}_n\le u_n$ on $\mathbb{R}$. Moreover, \begin{equation}\label{more-}\|I_{\Omega_n^-}^{m_n}f\|_{{B}_{pq}^{s,w}(\mathbb{R}^N)}\lesssim C_{\Omega_n^-}^{m_n}\|f\|_{{B}_{pq}^{s+m_n,u}(\mathbb{R}^N)},\qquad\textrm{where }\ C_{\Omega_n^-}^{m_n}:=\sup_{d\in\mathbb{N}_0}\bigl[\mathscr{M}_{\Omega_n^-}^{m_n}(d)+\mathscr{N}_{\Omega_n^-}^{m_n}({d+1})\bigr].\end{equation}
{\rm (ii)} If $I_{\Omega_n^-}^{m_n}f\in {A}_{pq}^{s,w}(\mathbb{R}^N)$ then $f\in {A}_{pq}^{s-m_n,w}(\mathbb{R}^N)$, besides,
$$\|f\|_{{A}_{pq}^{s-m_n,w}(\mathbb{R}^N)}\lesssim\|I_{\Omega_n^-}^{m_n}f\|_{{A}_{pq}^{s,w}(\mathbb{R}^N)}.$$
\end{theorem}

\begin{remark} The case $0<p\le 1$ can be also involved into consideration in Theorems \ref{ImagesA} and \ref{ImagesB} with properly modified conditions to be satisfied instead of $C_{\Omega_l^+}^{m_l}<\infty$ and $C_{\Omega_n^-}^{m_n}<\infty$, respectively (see \cite[\S\,1.4]{PSU} and \cite[Chapter 11, Section 1.5, Theorem 4]{KA}.
\end{remark}

\begin{remark} By the cost of evaluation constants in \eqref{more+} and \eqref{more-}, depending on fixed number parameters only, one can reduce the conditions \eqref{Md+} and \eqref{Nd+} (or \eqref{Md-} and  \eqref{Nd-}) to \eqref{Nd+} or \eqref{Nd-}, respectively.
\end{remark}

\begin{remark}
The conditions \eqref{Md+} (or \eqref{Md-}) and \eqref{Nd+} (or \eqref{Nd-}) can be significantly simplified in some special cases. In particular, if $p=2$ and $\int_{Q_{d\tau}}v_k\approx\,2^{-d}v_k(\tau/2^d)$ for the both $v_k=u_k$ and $v_k=w_k$, then 
$$\mathscr{M}_{\Omega_k^\pm}^{m_k}(d)\equiv \mathscr{M}_{\Omega_k^\pm}^{m_k}(0)\quad\textrm{and}\quad\mathscr{N}_{\Omega_k^\pm}^{m_k}({d+1})\equiv \mathscr{N}_{\Omega_k^\pm}^{m_k}(1)\qquad\textrm{for any }d\in\mathbb{N}_0,\qquad(k\in\{l,n\}).$$
\end{remark}

Our main result, which is connecting images and pre--images of integration operators $I_{\Omega^\star}^{|m|}$ of orders $m=(m_1,\ldots,m_N)$, follows by combinations of Theorems \ref{ImagesA} and \ref{ImagesB} and reads
\begin{theorem}\label{ImagesComb} Let $1<p<\infty$, $0<q\le\infty$, $s\in\mathbb{R}$, weights $u,w\in\mathscr{A}_\infty^\loc$ and $f\in L_1^\textrm{loc}(\mathbb{R}^N)$, $\mathrm{supp}\,f\subset\Omega^\star$. Let $I_{\Omega^\star}^{|m|}$ be an operator of order $m=(m_1,\ldots,m_N)$ produced by $r_m$ combinations of Riemann--Liouville operators of the forms \eqref{left} and/or \eqref{right}, where $r_m:=\mathrm{card}\bigl\{m_l\in \{m_1,\ldots,m_N\}\colon m_l\not=0,\ l\in\{1,\ldots,N\}\bigr\}$.\\ {\rm (i)} Assume $u,w$ are of product type, that is $v(x_1,\ldots,x_N)=v_1(x_1)\cdot\ldots\cdot v_N(x_N)$ for the both $v=u$ and $v=w$ with one variable functions $v_l(x_l)$, $l=1,\ldots,N$. Then $I^{|m|}_{\Omega^\star} f\in {B}_{pq}^{s,w}(\mathbb{R}^N)$ if $f\in {B}_{pq}^{s+|m|,u}(\mathbb{R}^N)$ and the following conditions are satisfied: $$\mathscr{M}_{\Omega^\star}^{|m|}(d)+\mathscr{N}_{\Omega^\star}^{|m|}(d+1):=\sum_{k\in\{1,\ldots,N\}:\,m_k\not=0}\Bigl[\mathscr{M}_{\Omega^\pm_k}^{m_k}(d)+\mathscr{N}_{\Omega^\pm_k}^{m_k}(d+1)\Bigr]<\infty \quad\textrm{for all }\ d\in\mathbb{N}_0.$$ Moreover,
$$\|I^{|m|}_{\Omega^\star} f\|_{{B}_{pq}^{s,w}(\mathbb{R}^N)}\lesssim C^{|m|}_{\Omega^\star} \|f\|_{{B}_{pq}^{s+|m|,u}(\mathbb{R}^N)},\qquad\textrm{where }\ C^{|m|}_{\Omega^\star}:=\sup_{d\in\mathbb{N}_0}\bigl[\mathscr{M}_{\Omega^\star}^{|m|}(d)+\mathscr{N}_{\Omega^\star}^{|m|}(d+1)\bigr].$$
{\rm (ii)} If $I^{|m|}_{\Omega^\star} f\in {A}_{pq}^{s,w}(\mathbb{R}^N)$ then $f\in {A}_{pq}^{s-|m|,u}(\mathbb{R}^N)$, besides,
$$\|f\|_{{A}_{pq}^{s-|m|,w}(\mathbb{R}^N)}\lesssim\|I^{|m|}_{\Omega^\star}f\|_{{A}_{pq}^{s,w}(\mathbb{R}^N)}.$$
\end{theorem}

Integration operators $I_\star^{|m|}$ produced by combinations of Riemann--Liouville operators of the forms
\begin{equation}\label{left'}
I_{+}^{m_l} f({y}_{x_l}):=\frac{1}{\Gamma(m
_l)}\int_{-\infty}^{x_l}\frac{f({y})}{(x_l-y_l)^{1-m_l}}\,dy_l,\qquad x_l\in\mathbb{R},
\end{equation} and
\begin{equation}\label{right'}
I_{-}^{m_n} f({y}_{x_l}):=\frac{1}{\Gamma(m_n)}\int_{x_n}^{+\infty}\frac{f({y})}{(y_n-x_n)^{1-m_n}}\,dy_n,\qquad x_n\in\mathbb{R},
\end{equation} can be also taken into consideration provided $I_{+}^{m_l}f\in L_1^\mathrm{loc}(\mathbb{R}^N)$ and $I_{-}^{m_n}f\in L_1^\mathrm{loc}(\mathbb{R}^N)$, respectively. Symbol $\star=(\star_1,\ldots,\star_N)$ in $I_\star^{|m|}$ has the same meaning as before and is related to $m=(m_1,\ldots,m_N)$.
\begin{theorem}\label{ImagesCombRem} Let $1<p<\infty$, $0<q\le\infty$, $s\in\mathbb{R}$, weights $u,w\in\mathscr{A}_\infty^\loc$ and $f\in L_1^\textrm{loc}(\mathbb{R}^N)$. Let $I_{\star}^{|m|}$ be an integration operator of order $m=(m_1,\ldots,m_N)$ produced by $r_m$ combinations of operators of the forms \eqref{left'} and/or \eqref{right'}. Assume that $I_{\star}^{|m|}f\in L_1^\textrm{loc}(\mathbb{R}^N)$.\\ {\rm (i)} Suppose $u,w$ are of product type, that is $v(x_1,\ldots,x_N)=v_1(x_1)\cdot\ldots\cdot v_N(x_N)$ for the both $v=u$ and $v=w$ with one variable functions $v_l(x_l)$, $l=1,\ldots,N$. Then $I^{|m|}_{\star} f\in {B}_{pq}^{s,w}(\mathbb{R}^N)$ if $f\in {B}_{pq}^{s+|m|,u}(\mathbb{R}^N)$ and \begin{equation}\label{uslovie}\mathscr{M}_{\star}^{|m|}(d)+\mathscr{N}_{\star}^{|m|}(d+1):=\sum_{k\in\{1,\ldots,N\}:\,m_k\not=0}\Bigl[\mathscr{M}_{\pm}^{m_k}(d)+\mathscr{N}_{\pm}^{m_k}(d+1)\Bigr]<\infty \quad\textrm{for all }\ d\in\mathbb{N}_0,\end{equation} where $\mathscr{M}_{\pm}^{m_k}(d)=\mathscr{M}_{\Omega^\pm_k}^{m_k}(d)$ and $\mathscr{N}_{\pm}^{m_k}(d+1)=\mathscr{N}_{\Omega^\pm_k}^{m_k}(d+1)$ with supreme taken over all $\tau_k\in\mathbb{Z}$. Moreover,
$$\|I^{|m|}_{\star} f\|_{{B}_{pq}^{s,w}(\mathbb{R}^N)}\lesssim C^{|m|}_{\star} \|f\|_{{B}_{pq}^{s+|m|,u}(\mathbb{R}^N)},\qquad\textrm{where }\ C^{|m|}_{\star}:=\sup_{d\in\mathbb{N}_0}\bigl[\mathscr{M}_{\star}^{|m|}(d)+\mathscr{N}_{\star}^{|m|}(d+1)\bigr].$$
{\rm (ii)} If $I^{|m|}_{\star} f\in {A}_{pq}^{s,w}(\mathbb{R}^N)$ then $f\in {A}_{pq}^{s-|m|,u}(\mathbb{R}^N)$, besides,
$$\|f\|_{{A}_{pq}^{s-|m|,w}(\mathbb{R}^N)}\lesssim\|I^{|m|}_{\star}f\|_{{A}_{pq}^{s,w}(\mathbb{R}^N)}.$$
\end{theorem}

\section{Entropy and approximation numbers}\label{EnAp}

For $k\in\mathbb{N}$ and a linear continuous operator $T$ between quasi--Banach spaces $A_1$ and $A_2$ \cite{CS, Pi}:\\ -- the $k$--th  (dyadic) entropy number $e_k(T)$ of $T\colon A_1\rightarrow A_2$ is the infimum of all numbers $\varepsilon>0$ such that there exist $2^{k-1}$ balls in $A_2$ of radius $\varepsilon$ which cover $TU_1$, where $U_1$ is the unit ball in $A_1$;\\
-- the $k$--th approximation number $a_k(T)$ of $T\colon A_1\rightarrow A_2$ is the infimum of all numbers $\|T-L\|$, where $L$ runs through the collection of all continuous linear mappings from $A_1$ to $A_2$ with ${\rm rank}\, L<k$.

Let $1<p<\infty$, $0<q\le\infty$, $s\in\mathbb{R}$ and $w\in\mathscr{A}_\infty^\loc$. Let $I_{\star}^{|m|}$ be an integration operator of order $m=(m_1,\ldots,m_N)$ produced by $r_m$ combinations of operators of the forms \eqref{left'} and/or \eqref{right'}. Under $A_1$ and $A_2$ we shall understand weighted Besov and Triebel--Lizorkin spaces $A_{pq}^{s,w}(\mathbb{R}^N)\subset L_1^\textrm{loc}(\mathbb{R}^N)$.

Our results in this part of the article connect entropy and approximation numbers
 of operators
$$I^{|m|}_{{\star}}\colon {A}_{p_1q_1}^{s_1,v}(\mathbb{R}^N)\rightarrow {A}_{p_2q_2}^{s_2,w}(\mathbb{R}^N)$$
with the same characteristics of embeddings $id$ of Besov and Triebel--Lizorkin classes ${A}_{pq}^{s,w}(\mathbb{R}^N)$.

\begin{theorem}\label{ImagesNumb} Let $1<p_1,p_2<\infty$, $0<q_1,q_2\le\infty$, $-\infty<s_1,s_2<+\infty$ and factorisable weights $u,v,w$ be of $\mathscr{A}_\infty^\loc$ type. For a given $N\in\mathbb{N}$ let $m=(m_1,\ldots,m_N)$ be a multi--index, $m_k\in\mathbb{N}_0$, $k=1,\ldots,N$. Assume that $B_{p_1q_1}^{s_1,v}(\mathbb{R}^N)$, $B_{p_2q_2}^{s_2+|m|,u}(\mathbb{R}^N)$ and $B_{p_2q_2}^{s_2-|m|,w}(\mathbb{R}^N)$ are subspaces of $L_1^\textrm{loc}(\mathbb{R}^N)$. Suppose that $I^{|m|}_{\star}\colon {B}_{p_1q_1}^{s_1,v}(\mathbb{R}^N)\hookrightarrow {B}_{p_2q_2}^{s_2,w}(\mathbb{R}^N)$ is an integration operator of order $m$ produced by $r_m$ combinations of Riemann--Liouville operators of the forms \eqref{left'} and/or \eqref{right'}. 
Let $$id_u\colon {B}_{p_1q_1}^{s_1,v}(\mathbb{R}^N)\hookrightarrow {B}_{p_2q_2}^{s_2+|m|,u}(\mathbb{R}^N), \qquad id_w\colon {B}_{p_1q_1}^{s_1,v}(\mathbb{R}^N)\hookrightarrow {B}_{p_2q_2}^{s_2-|m|,w}(\mathbb{R}^N)$$ be the embedding operators. It holds in the case $s_2-|m|\le s_1$ $$e_k(I^{|m|}_{\star})\gtrsim e_k(id_w)\quad(k\in\mathbb{N}),\qquad a_k(I^{|m|}_{\star})\gtrsim a_k(id_w)\quad(k\in\mathbb{N}).$$
If $s_2+|m|\le s_1$ and the conditions \eqref{uslovie} are satisfied then $$e_k(I^{|m|}_{\star})\lesssim e_k(id_u)\quad(k\in\mathbb{N}),\qquad a_k(I^{|m|}_{\star})\lesssim a_k(id)\quad(k\in\mathbb{N}).$$

\end{theorem}
\begin{proof} Continuity and compactness criteria for embeddings $id$ of Besov and Triebel--Lizorkin spaces with weights were derived in e.g. \cite{HTr2,HSc}. In particular, assumptions given in \cite[Proposition 2.1]{HSc} implies the restrictions $-\infty<s_2\pm|m|\le s_1<\infty$ on smoothness parameters $s_1$ and $s_2$ in our cases.

Denote $$R^\ast\colon {B}_{p_2q_2}^{s_2+|m|,u}(\mathbb{R}^N)\rightarrow {B}_{p_2q_2}^{s_2,w}(\mathbb{R}^N)$$ and
$$R_\ast\colon {B}_{p_2q_2}^{s_2,w}(\mathbb{R}^N)\rightarrow {B}_{p_2q_2}^{s_2-|m|,w}(\mathbb{R}^N).$$ Then $I_{\star}^{|m|}=R^\ast\circ id_u$ and $id_w=R_\ast\circ I_{\Omega^\star}^{|m|}$, where $R^\ast$, $R_\ast$ are bounded on the strength of Theorem \ref{ImagesCombRem}. From here the result follows directly from multiplicativity properties of entropy and approximation numbers \cite{CS, Pi}.
\begin{displaymath}
                    \xymatrix{
{B}_{p_1q_1}^{s_1,v}(\mathbb{R}^N) \ar[r]^{~~~~~I_{\star}^{|m|}~~} \ar[rd]_{id_u} &
{B}_{p_2q_2}^{s_2,w}(\mathbb{R}^N)\\
&{B}_{p_2q_2}^{s_2+|m|,u}(\mathbb{R}^N) \ar[u]_{R^\ast}}
\qquad
\xymatrix{
{B}_{p_1q_1}^{s_1,v}(\mathbb{R}^N) \ar[r]^{~~~~~I_{\star}^{|m|}~~} \ar[rd]_{id_w} &
{B}_{p_2q_2}^{s_2,w}(\mathbb{R}^N) \ar[d]^{{R_\ast}}\\
&{B}_{p_2q_2}^{s_2-|m|,w}(\mathbb{R}^N) }
\end{displaymath}
\end{proof}

Estimates for the entropy $e_k(id)$
and approximation numbers $a_k(id)$
of embeddings $id$ of Besov and Triebel--Lizorkin spaces with Muckenhoupt weights of purely polynomial growth can be found in \cite{HSc}. By Theorem \ref{ImagesNumb}, these results (see \cite[\S\S\,3,4]{HSc}) can be directly applied for obtaining explicit estimates on $e_k(I^{|m|}_{{\star}})$ and $a_k(I^{|m|}_{{\star}})$ when $B_{p_1q_1}^{s_1,v}(\mathbb{R}^N)\subset L_1^\mathrm{loc}(\mathbb{R}^N)$, $B_{p_2q_2}^{s_2+|m|,u}(\mathbb{R}^N)\subset L_1^\textrm{loc}(\mathbb{R}^N)$ and $B_{p_2q_2}^{s_2-|m|,w}(\mathbb{R}^N)\subset L_1^\textrm{loc}(\mathbb{R}^N)$. Related conditions on $s$ for $A_{pq}^{s}(\mathbb{R}^N,w)\subset L_1^\mathrm{loc}(\mathbb{R}^N)$ to hold can be seen in e.g. \cite[\S\,2]{HSc} and \cite[\S\,5]{Ma}. For the case of so called admissible weight functions $u,v$ and $w$ one can consult \cite{HTr2}. 

\begin{remark} The lower estimates for the entropy and approximation numbers of integration operators in
Triebel--Lizorkin spaces $F_{pq}^{s,w}(\mathbb{R}^N,w)\subset L_1^\textrm{loc}(\mathbb{R}^N)$ follow analogously to the result of Theorem \ref{ImagesNumb}. Upper estimates can be obtained by the chain of embeddings $B_{pq_\ast}^{s,w}(\mathbb{R}^N)\hookrightarrow F_{pq}^{s,w}(\mathbb{R}^N) \hookrightarrow B_{pq^\ast}^{s,w}(\mathbb{R}^N)$, $q_\ast\le\min\{p,q\}\le\max\{p,q\}\le q^\ast$, and basing on the result of Theorem \ref{ImagesNumb} for Besov spaces. 
\end{remark}

Method of reducing the study of behaviour of the entropy and approximation numbers of Hardy integral operators to the same problem for embedding operators was also exploit in \cite{NU}. Some other results related to characteristic numbers of integral operators may be found in \cite{EdLg1, EdLg2, Lg1,Lg2,LNM,LL, LS}.

\end{document}